\documentclass[a4paper,11pt,reqno,oneside]{amsart}

\usepackage[backend=bibtex, giveninits, style=alphabetic, uniquelist=false, doi=false, url=false, isbn = false]{biblatex}
\renewcommand{\leq}{\leqslant}
\renewcommand{\geq}{\geqslant}
\usepackage{amssymb}

\newcommand{\dd}{\mathrm{d}}

\newcommand{\Ec}{\mathcal E}

\newcommand{\Hc}{\mathcal H}

\newcommand{\Nc}{\mathcal N}

\newcommand{\Rc}{\mathcal R}

\newcommand{\Tc}{\mathcal T}

\newcommand{\LL}{\mathbb L}

\newcommand{\N}{\mathbb N}

\newcommand{\R}{\mathbb R}

\newcommand{\Z}{\mathbb Z}

\newcommand{\Nast}{\N^\ast}

\newcommand{\Eb}{\mathbf{E}}

\newcommand{\Pb}{\mathbf{P}}

\usepackage{mathrsfs}

\usepackage{dsfont}
\newcommand{\1}{\mathds 1}

\renewcommand{\epsilon}{\varepsilon}
\renewcommand{\phi}{\varphi}

\newcommand{\argmin}{\operatornamewithlimits{argmin}}
\newcommand{\argmax}{\operatornamewithlimits{argmax}}
\newcommand{\KL}{\mathrm{KL}}
\newcommand{\TV}{\mathrm{TV}}
\newcommand{\var}{\operatorname{Var}}

\newcommand{\supp}{\mathrm{supp}}

\newcommand{\defeq}{\vcentcolon=}
\newcommand{\eqdef}{=\vcentcolon}

   \newcommand{\abs}[1]{\vert#1\vert}

\usepackage{mathtools} \renewcommand{\complement}{\mathsf{c}}

\newcommand{\sol}{\theta}
\newcommand{\soltilde}{\widetilde{\sol}}
\newcommand{\solcirc}{\sol^\circ}

\newcommand{\ev}{\lambda}
\newcommand{\Op}{A}

\newcommand{\csol}{\Theta}
\newcommand{\cev}{\Ec}
\newcommand{\qfunc}{\mathfrak q}
\newcommand{\qhat}{{\widehat{\mathfrak{q}}}}
\newcommand{\qtilde}{{\widetilde{\qfunc}}}
\newcommand{\setAlt}{\Theta_1}

\newcommand{\setpm}{\{ \pm 1 \}}

\newcommand{\Yp}{\widetilde Y}
\newcommand{\Ypp}{\bar{Y}}

\newcommand{\kepsilon}{{k_\epsilon}}
\newcommand{\ksigma}{{k_\sigma}}
\newcommand{\kstar}{k^\star}

\newcommand{\qfuncsd}{{\qfunc}^{\mathrm{sd}}}
\newcommand{\qfuncgof}{{\qfunc}^{\mathrm{gof}}}
\newcommand{\ksd}{\kappa_1}
\newcommand{\qhatksd}{\qhat_{\ksd}}
\newcommand{\kgof}{\kappa_2}
\newcommand{\qhatkgof}{\qhat_{\kgof}}

\newcommand{\Deltahat}{\widehat{\Delta}}
\newcommand{\Deltahatsd}{\widehat{\Delta}^{\mathrm{sd}}}
\newcommand{\Deltahatgof}{\widehat{\Delta}^{\mathrm{gof}}}
\newcommand{\Ctilde}{\widetilde{C}}
\newcommand{\ctilde}{\widetilde{c}}\usepackage{lmodern}
\usepackage[x11names]{xcolor}
\usepackage[english]{babel}

\usepackage{amsmath}
\usepackage{amsfonts}
\usepackage{amssymb}
\usepackage{amsthm}
\usepackage{amscd}
\usepackage{amsaddr}

\usepackage[paper=a4paper,left=30mm,right=30mm,top=30mm,bottom=30mm]{geometry}

\usepackage{bookmark} 

\usepackage{csquotes}

\usepackage{hyperref}
\hypersetup{colorlinks=true,urlcolor=Blue1,linkcolor=Blue1,citecolor=Blue1,naturalnames=true,hypertexnames=false}

\usepackage{marvosym} 

\usepackage{mathtools} 

\usepackage[shortlabels]{enumitem}

\usepackage{stmaryrd} 

\usepackage{tabularx}
\usepackage[font=footnotesize]{caption}
\usepackage{colortbl}

\usepackage[notref, notcite, color, final]{showkeys}
\definecolor{refkey}{rgb}{0.0,0.0,0.0}
\definecolor{labelkey}{rgb}{0.0,0.0,0.0}

\theoremstyle{plain}
\newtheorem{theorem}{Theorem}[section]
\newtheorem{proposition}[theorem]{Proposition}
\newtheorem{lemma}[theorem]{Lemma}
\newtheorem{corollary}[theorem]{Corollary}

\theoremstyle{definition}
\newtheorem{definition}[theorem]{Definition}
\newtheorem{assumption}[theorem]{Assumption}

\theoremstyle{remark}
\newtheorem{remark}[theorem]{Remark}

\DeclareFieldFormat[article]{title}{#1} \DeclareFieldFormat[book]{title}{\itshape #1}
\DeclareFieldFormat[thesis]{title}{#1}
\DeclareFieldFormat[incollection]{title}{\itshape #1}
\DeclareFieldFormat[inproceedings]{title}{\itshape #1}
\DeclareFieldFormat[article]{author}{\color{blue} #1}
\DeclareFieldFormat[article]{pages}{#1.}
\DeclareFieldFormat[incollection]{pages}{#1.} 

\renewbibmacro{in:}{\ifentrytype{article}{}{\printtext{\bibstring{in}\intitlepunct}}}

\DeclareNameAlias{sortname}{last-first}
\DeclareNameAlias{default}{last-first}

\renewbibmacro*{volume+number+eid}{\printfield{volume}\setunit*{\addnbthinspace}\printfield{number}\setunit{\addcomma\space}\printfield{eid}}
\DeclareFieldFormat[article]{number}{\mkbibparens{#1}}

\renewbibmacro*{publisher+location+date}{\printlist{publisher}\iflistundef{location}
	{\setunit*{\addcomma\space}}
	{\setunit*{\addcomma\space}}\printlist{location}\setunit*{\addcomma\space}\usebibmacro{date}\newunit}

 \DeclareFieldFormat[article]{number}{}

\addto{\captionsenglish}{}

\newenvironment{smalign*}
 {\par$\!\aligned}
 {\endaligned$\par}
 
\allowdisplaybreaks %

\title[Quadratic functionals in inverse problems with unknown operator]{Rate optimal estimation of quadratic functionals in inverse problems with partially unknown operator and application to testing problems}
\author{Martin Kroll}
\address{CREST, ENSAE, Institut Polytechnique de Paris}
\email{martin.kroll@ensae.fr}

\thanks{Ce travail a \'et\'e r\'ealis\'e dans le cadre du laboratoire d'excellence ECODEC, portant la r\'ef\'erence ANR-11-LABX-0047.}
\date{\today}
\subjclass[2010]{62G05 (primary), and 62G10 (secondary)}
\keywords{Inverse problem, unknown eigenvalues, minimax theory, rate optimality, truncated series estimator, non-parametric testing, goodness-of-fit testing, signal detection.}

\begin{document}

\begin{abstract}
	We consider the estimation of quadratic functionals in a Gaussian sequence model where the eigenvalues are supposed to be unknown and accessible through noisy observations only.
	Imposing smoothness assumptions both on the signal and the sequence of eigenvalues, we develop a minimax theory for this problem.
	We propose a truncated series estimator and show that it attains the optimal rate of convergence if the truncation parameter is chosen appropriately.
	Consequences for testing problems in inverse problems are equally discussed: in particular, the minimax rates of testing for signal detection and goodness-of-fit testing are derived.
\end{abstract} 
\maketitle

\section{Introduction}

\subsection{Problem statement}
We consider the Gaussian sequence model
\begin{align}
X_j &= \ev_j \sol_j + \epsilon \xi_j , \quad j \in \Nast = \{ 1,2,\ldots \}\label{eqObsX}, \quad \text{and}\\
Y_j &= \ev_j + \sigma \eta_j, \quad j \in \Nast\label{eqObsY}
\end{align}
where $(\xi_j)_{j \in \Nast}$, $(\eta_j)_{j \in \Nast}$ are independent random vectors with independent standard Gaussian components and $\epsilon, \sigma \in (0,1)$ are known noise levels.
Given some known and fixed 'reference point' $(\solcirc_j)_{j \in \Nast}$, we will in this work address the following two questions:
\begin{enumerate}\item Let $(\omega_j)_{j \in \Nast}$ be some known sequence of weights. How can we estimate the value of the quadratic functional
 \begin{equation*}
	\qfunc( \theta ) = \sum_{j=1}^\infty \omega_j^2 (\sol_j - \solcirc_j)^2
\end{equation*}
from data $(X_j)_{j \in \Nast}$ and $(Y_j)_{j \in \Nast}$ in an optimal way?
\item How can we test the null hypothesis $\sol = \solcirc$ against the alternative $\sol \in \setAlt$ for some $\setAlt$ with $\solcirc \notin \setAlt$?
\end{enumerate}
Concerning both questions, the sequence $(\lambda_j)_{j\in \Nast}$ is a nuisance parameter and only accessible by means of the observations $(Y_j)_{j \in \Nast}$.
Specific choices include both the case $\lambda_j \equiv 1$ (then, \eqref{eqObsX} is the classical Gaussian sequence model with \emph{direct} observations) and the case $\lambda_j \to 0$ making the inverse problem of reconstructing $\sol$ \emph{ill-posed} (see \cite{cavalier2011inverse}, Definition~1.1\ for a definition of well-/ill-posedness).
Precise assumptions on all model parameters will be given in Section~\ref{secMeth} below.
To the best of our knowledge, the model given by \eqref{eqObsX} and \eqref{eqObsY} was introduced explicitly in~\cite{cavalier2005adaptive} for the first time, and is also referred to as an inverse problem with \emph{partially unknown operator} \cite{johannes2013partially,marteau2017minimax}.
In the context of inverse problems, in its general form given by an operator equation $X = \Op \sol + \epsilon \xi$, this model provides something between the classical assumption that the operator $\Op$ is known~\cite{donoho1995nonlinear,cavalier2011inverse} and the assumption that the operator is only accessible by a blurred observation $Y = \Op + \sigma \Xi$~\cite{efromovich2001inverse,hoffmann2008nonlinear}:
it arises by the structural assumption that the operator $\Op^\ast\Op$ is diagonal with eigenvalues $(\lambda_j^2)_{j \in \Nast}$.
We refer to the references~\cite{cavalier2005adaptive,johannes2013partially,marteau2017minimax} for a more detailed derivation and further motivation of the model.
Note that, whereas the non-parametric estimation of the parameter $\sol$ itself from observations~\eqref{eqObsX} and \eqref{eqObsY} (including adaptation) was intensively studied in \cite{cavalier2005adaptive,johannes2013partially}, the estimation of quadratic functionals has not yet been considered, and also the question of non-parametric testing has been investigated only recently (see the following Subsection~\ref{subsRelatedWork} for a discussion of related work).

\subsection{Related work}\label{subsRelatedWork}
Starting with the paper~\cite{bickel1988estimating}, the estimation of quadratic functionals has received a lot of attention in non-parametric statistics, in particular in models with direct observations~\cite{donoho1990minimax,fan1991estimation,gayraud1999wavelet,laurent2000adaptive,johnstone2001chi-square,johnstone2001thresholding,cai2005nonquadratic,laurent2005adaptive,cai2006optimal,klemela2006sharp,gine2008simple,rivoirard2008maxiset,collier2017minimax}.
In the context of statistical inverse problems, there is much less work dedicated to this problem.
\cite{butucea2007goodness} provides a goodness-of-fit test in a convolution model where the test statistic is based on the estimator of a quadratic functional.
The paper~\cite{butucea2011quadratic} considers observations as in~\eqref{eqObsX} but assumes the sequence of eigenvalues to be known.
Under this assumption, minimax upper bounds in terms of $\epsilon$ are derived for both ordinary smooth and supersmooth $\sol$.
In addition, the authors assume that their approach even provides optimal constants.
\cite{chesneau2011adaptive} considers adaptive estimation of the $\LL^2$-norm in a model where a convolution product of an
unknown function and a known function is corrupted by Gaussian noise.

The estimation of quadratic functionals is closely related to hypotheses testing since estimators of quadratic functionals provide natural building blocks for test statistics.
Starting with the seminal paper~\cite{ermakov1990minimax}, the theory of non-parametric testing in direct Gaussian sequence space models has been rigorously developed in a series of papers by Ingster~\cite{ingster1993asymptotically} (see also the monograph~\cite{ingster2003nonparametric}). 
In the domain of inverse problems, an increasing interest in theoretical results in the spirit of the book~\cite{ingster2003nonparametric} has arisen within the last decade~\cite{laurent2011testing,ingster2012minimax,marteau2015unified,marteau2017minimax}, partially motivated by applications coming from biology~\cite{bissantz2009testing} or astrophysics~\cite{lacour2014goodness}.
However, concerning inverse problems with partially unknown operator, the existing research literature reduces to the paper~\cite{marteau2017minimax} that considers the same model as in the present work.
In contrast to our approach in Section~\ref{secTesting} where we use the sum of type I and type II error in order to measure the performance of tests, the authors of \cite{marteau2017minimax} consider level-$\alpha$-tests (i.e., tests whose type I error is bounded from above by some prespecified $\alpha \in (0,1)$) and try to minimize the type II error under this constraint.
In this framework, the authors derive upper and lower bounds for the so-called \emph{separation rate}.
Their test statistic is also based on estimation of a quadratic functional but only the goodness-of-fit testing problem is considered.
The authors obtain a slight gap by a logarithmic factor between upper and lower bounds with respect to the noise level $\sigma$.
A main difference between the present paper and~\cite{marteau2017minimax} concerns the minimax methodology: \cite{marteau2017minimax} impose a smoothness condition on the sequence $(\lambda_j)_{j \in \Nast}$ (equivalent to our one introduced in Section~\ref{secMeth}) only in order to establish lower bounds, but the construction of their test statistic is independent of this smoothness.
Thus, their testing procedure is adaptive with respect to the sequence $(\lambda_j)_{j \in \mathbb N}$, whereas we assume the order of the decay of this sequence to be known.
Imposing this additional assumption, we are able to derive upper and lower bounds for the testing rate that match (without any logarithmic gap).
It might be of interest to explore to what extent the extra logarithmic factors in \cite{marteau2017minimax} might be unevitable in the adaptive scenario.
However, answering this question is outside the scope of this work and deferred to future research.

\subsection{Organisation and main contributions of the paper}
Let us summarize the main contributions of this paper. We emphasize in advance that all results of the paper are non-asymptotic.
\begin{itemize}
	\item We introduce a truncated series estimator of $\qfunc(\sol)$ (Section~\ref{secMeth}), and derive minimax upper bounds for this estimator in terms of the noise levels $\epsilon$ and $\sigma$ (Section~\ref{secUpper}).
	The construction of the estimator is based on a \emph{sample cloning} technique that has not been used before to construct estimators of quadratic functionals.
	\item We prove minimax lower bounds for the estimation of $\qfunc(\sol)$ from data \eqref{eqObsX} and \eqref{eqObsY}.
	These results show that the truncated series estimators is rate optimal provided that the truncation parameter is chosen appropriately.
	\item Our abstract results indicate an 'elbow effect' of the optimal rate of convergence in terms of the noise level $\sigma$ that is similar to the well-known elbow effect in $\epsilon$ \cite{fan1991estimation}.
	However, the rate in $\sigma$ is in general faster than the one in $\epsilon$ and the parametric rate $\sigma^2$ can be attained in cases where the non-parametric regime holds with respect to $\epsilon$.
	For instance, in the case that the signal belongs to a Sobolev class of index $p$ and the considered inverse problem is mildly ill-posed with degree of ill-posedness equal to $a$, the optimal rate of convergence will turn out to be
	\begin{equation*}
		\epsilon^2 \vee \epsilon^{16p/(4a+4p+1)} \vee \sigma^2 \vee \sigma^{4p/a}.
	\end{equation*}
	\item In Section~\ref{secTesting}, as a rather direct application of our results on the estimation of quadratic functionals we consider non-parametric testing problems of the type
	\begin{equation*}
		\Hc_0\colon \sol = \solcirc \in \csol \qquad \text{against} \qquad \Hc_1\colon \sol \in \csol, \Vert \sol - \solcirc \Vert_2 \geq r
	\end{equation*}
	for some $r > 0$.
	As already remarked by Marteau and Sapatinas~\cite{marteau2017minimax}, the case of signal detection ($\solcirc = 0$, Section~\ref{ssSignal}) and the one of goodness-of-fit testing $(\solcirc \neq 0$, Section~\ref{ssGoF}) have to be treated separately.
	For both problems, we derive the minimax rate of testing and propose a test statistic attaining this rate.
	In particular, in coincidence with the findings in~\cite{laurent2011testing}, it turns out that for the signal detection problem it is sufficient to consider the observation~\eqref{eqObsX} and construct a test statistic based on an estimator of a quadratic functional of $\soltilde = \ev \sol$.
	For the goodness-of-fit problem, however, the testing rate depends also on $\sigma$ and both observations, $X$ and $Y$, are taken into account for the construction of the test statistic.
\end{itemize} \section{Methodology}\label{secMeth}

\subsection{Notation} 
We frequently denote entire sequences by single letters when writing 'the sequence $a$' instead of 'the sequence $(a_j)_{j \in \Nast}$'.
Numerical operations on sequences like $a^{-1}$ are to be understood elementwise.
Throughout $C$ denotes a purely numerical constant and $C(\ldots)$ a constant that depends only on the parameters indicated within parentheses.
$x \lesssim y$ is shorthand for $x \leq Cy$, and we write $x \asymp y$ if $x \lesssim y$ and $y \lesssim x$ hold simultaneously.
Moreover, $x \asymp_\nu y$ means that $x\nu^{-1} \leq y \leq \nu x$.
We put $\llbracket x,y\rrbracket = [x,y] \cap \Z$ for $x,y \in \R$.

\subsection{Truncated series estimator}\label{subsTSE}

In order to define a truncated series estimators, we first generate two independent instances of the $Y$ sample by the following \emph{sample cloning} technique which is well-known in the context of aggregation (see~\cite{tsybakov2014aggregation}, Lemma~2.1): let $\widetilde{\eta}$ be a sequence of independent standard Gaussian random variables independent of $\xi$ and $\eta$.
For $j \in \Nast$, we put
\begin{equation*}
	\Yp_j = Y_j + \sigma \widetilde \eta_j \qquad \text{and} \qquad \Ypp_j = Y_j - \sigma \widetilde \eta_j.
\end{equation*}
Then $\Yp_j, \Ypp_j$ are i.i.d. $\Nc(\lambda_j, 2\sigma^2)$, and the price to pay for the availability of two independent samples is a doubling of the variance.
Based on the availability of the samples $\Yp = (\Yp_j)_{j \in \Nast}$, $\Ypp = (\Ypp_j)_{j \in \Nast}$ we define, for any $k \in \Nast$, the truncated series estimator
\begin{equation}\label{eqDefEstSeries}
\qhat_k = \sum_{j=1}^{k} \omega_j^2 \, \frac{U_j}{V_j} \1_{\Omega_j}
\end{equation}
where $U_j \defeq (X_j - \Yp_j\solcirc_j)^2 - \epsilon^2 - 2 (\solcirc_j)^2 \sigma^2$, $V_j \defeq \Ypp_j^2 - 2\sigma^2$ and $\Omega_j \defeq \{\Ypp_j^2 \geq 3\sigma^2\}$.
Note that $U_j$ and $V_j$ are unbiased estimators of $\lambda_j^2 (\sol_j-\solcirc_j)^2$ and $\lambda_j^2$, respectively, guaranteeing that the fraction $U_j/V_j$ is at least a consistent estimator of $(\sol_j-\solcirc_j)^2$.
In addition, due to the construction based on sample cloning, $U_j$ and $\1_{\Omega_j}/V_j$ are independent as they can be written as $U_j=f(X_j, \Yp_j)$ and $\1_{\Omega_j}/V_j=g(\Ypp_j)$ for non-random functions $f$ and $g$.
Inspired by~\cite{neumann1997effect}, the additional cut-off $\1_{\Omega_j}$ in \eqref{eqDefEstSeries} excludes too small values of $V_j$ that would otherwise lead to an unstable behaviour of the entire estimator.
As usual in non-parametric statistics, the value of the truncation parameter $k \in \Nast$ has to be chosen by the statistician and crucially effects the performance of the estimator.
In Section~\ref{secUpper}, we first derive an upper risk bound for $\qhat_k$ that holds for any $k \in \Nast$, and then take the minimizer of this bound to define our final estimator.
This specific choice will turn out to define a rate optimal estimator under mild assumptions (of course, the resulting estimator is not adaptive).
Let us note that, in order to derive a minimax optimal estimator only, other truncated series estimators could have been chosen.
The construction of our estimator, however, is motivated by our application to testing in Section~\ref{secTesting}.

\subsection{Minimax estimation}\label{subsMinimax}

Given sequences $\gamma$ and $\alpha$, let us define the $\ell^2$-ellipsoid
\begin{equation*}
	\csol = \csol(\gamma,L) = \bigg \{ \sol \in \ell^2 \colon \sum_{j=1}^\infty \gamma_j^2\sol_j^2 \leq L^2 \bigg \}
\end{equation*}
and the $\ell^2$-hyperrectangle
\begin{equation*}
  \cev = \cev(\alpha,d) = \big\{ \ev \in \ell^\infty \colon d^{-1}\alpha_j \leq \vert \lambda_j \vert \leq d \alpha_j \big\}.
\end{equation*}
We usually suppress the dependence of $\csol$ and $\cev$ on $\gamma, \alpha, L, d$ in the notation.
For the rest of the paper, we assume that $(\sol, \ev) \in \csol \times \cev$.

\begin{definition}[Minimax rate of estimation, minimax estimator]
	An estimator $\qhat$ of $\qfunc(\sol)$ \emph{attains the rate} $\psi_{\epsilon,\sigma}^2$ over the smoothness classes $\csol$ and $\cev$ if there exists a numerical constant $C > 0$ such that
	\begin{equation*}
		\sup_{\sol \in \csol} \sup_{\ev \in \cev} \Eb [ \left( \qhat - \qfunc(\sol) \right)^2 ] \leq C \psi_{\epsilon, \sigma}^{2}.
	\end{equation*}
	The rate $\psi_{\epsilon, \sigma}^2$ is called \emph{minimax optimal} if in addition
	\begin{equation}
		\inf_{\qtilde } \sup_{\sol \in \csol} \sup_{\ev \in \cev} \Eb [ \left( \qtilde - \qfunc(\sol) \right)^2 ] \geq c \psi_{\epsilon, \sigma}^{2}\label{eqMinimaxRiskLower}
	\end{equation}
	holds for some $c > 0$ where the infimum is taken over all estimators based on observations~\eqref{eqObsX} and~\eqref{eqObsY}.
\end{definition}
In this work, the minimax optimal rate is derived under the following assumption on the sequences $\alpha, \gamma$ and $\omega$.

\begin{assumption}\label{assSeq}
	The sequences $\alpha$ and $\omega \gamma^{-1}$ are non-increasing and normalized such that $\alpha_1=\gamma_1=\omega_1=1$.
\end{assumption}

Assumption~\ref{assSeq} is rather mild and satisfied by all the examples considered later.
The proof of Theorem~\ref{thmUpper} shows that $\omega_k^4\gamma_k^{-4}$ is the order of the squared bias of our estimator, and hence the convergence of $\omega\gamma^{-1}$ to zero ensures consistency as $\max \{\epsilon,\sigma \}$ tends to zero.
The following special choices of the sequences $\alpha$ and $\gamma$ satisfy Assumption~\ref{assSeq}, and will be used throughout the paper to illustrate the general results.
Concerning the sequence $\alpha$ we consider either
\begin{itemize}
	\item the case $\alpha_j \asymp j^{-a}$ for some $a > 0$ (the inverse problem is \emph{mildly ill-posed} and $a$ the \emph{degree of ill-posedness}), or
	\item the case $\alpha_j \asymp \exp(-ja)$ for some $a > 0$ (the inverse problem is \emph{severely ill-posed}).
\end{itemize}
Concerning the sequence $\gamma$ we consider either
\begin{itemize}
	\item the case $\gamma_j \asymp j^p$ for some $p > 0$ ($\csol$ is a \emph{Sobolev ellipsoid}), or
	\item the case $\gamma_j \asymp \exp(pj)$ for some $p > 0$ ($\csol$ is an \emph{ellipsoid of analytic functions}).
\end{itemize}
The same smoothness assumptions have equally been used for the purpose of illustration in~\cite{ingster2012minimax} and~\cite{marteau2017minimax} making our results directly comparable to the ones obtained in those papers.

\subsection{Minimax theory of testing}\label{subsMethTesting}
In Section~\ref{secTesting}, we consider the problem of testing the simple hypothesis $\sol = \solcirc$ against the composite alternative $\sol \in \setAlt$ with $\solcirc \notin \setAlt$ (more precisely, we test $(\sol,\ev) \in \{ \solcirc \} \times \cev$ against $(\sol,\ev) \in \setAlt \times \cev$).
Usually, the case $\solcirc = 0$ is referred to as \emph{signal detection} and the case $\solcirc \neq 0$ as \emph{goodness-of-fit testing}.
By definition, a \emph{test statistic} $\Delta$ is a $\{ 0,1 \}$-valued function based on the observations $(X,Y)$.
Its performance is measured by the sum of type I and maximal type II error, $\Pb_0(\Delta = 1) + \sup_{\sol \in \setAlt} \Pb_\sol(\Delta = 0)$, and the corresponding benchmark is the quantity
\begin{equation*}
	\inf_{\widetilde{\Delta}} \left\lbrace \Pb_0(\widetilde{\Delta} = 1) + \sup_{\sol \in \setAlt} \Pb_\sol(\widetilde{\Delta} = 0)\right\rbrace 
\end{equation*}
where the infimum is taken over all test statistics $\widetilde{\Delta}$.
It is well-known that, apart from smoothness assumptions, the null hypothesis $\solcirc$ must be separated from the alternative $\setAlt$ at least by a certain distance in order to make non-trivial testing possible.
In this spirit, we consider for $r > 0$ the testing problems
\begin{equation*}
	\Hc_0\colon \sol = \solcirc \in \csol \qquad \text{against} \qquad \Hc_1 \colon \sol - \solcirc \in \setAlt(r)
\end{equation*}
where $\setAlt(r) = \csol \cap \left\lbrace \sol \in \ell^2(\Nast) \colon \Vert \sol \Vert_2 \geq r \right\rbrace$.
Based on this definition of $\setAlt(r)$, we put
\begin{equation*}
	\Rc(r) =  \inf_{\widetilde{\Delta}} \left\lbrace \Pb_0(\widetilde{\Delta} = 1) + \sup_{\sol \in \setAlt(r)} \Pb_\sol(\widetilde{\Delta} = 0)\right\rbrace.
\end{equation*}
The central quantity of our interest is the \emph{minimax testing rate}.
\begin{definition}\label{defRateTesting}
	The quantity $\phi_{\epsilon, \sigma}^2 > 0$ is called \emph{minimax testing rate} if the following two conditions are fulfilled:
	\begin{enumerate}[(i)]
		\item\label{defRateTesting1} for any $\delta \in (0,1)$, there exists $C^\ast > 0$ such that for all $C > C^\ast$ it holds
		\begin{equation*}
			\Rc (C\phi_{\epsilon,\sigma}) \leq \delta,
		\end{equation*}
		\item\label{defRateTesting2} for any $\delta \in (0,1)$, there exists $C_\ast > 0$ such that for all $0 < c < C_\ast$ it holds
		\begin{equation*}
			\Rc (c\phi_{\epsilon,\sigma}) \geq 1 - \delta.
		\end{equation*}
	\end{enumerate} 
\end{definition}
Given this purely non-asymptotic definition, the strategy for deriving the minimax testing rate is as follows: in order to prove the upper bound given by Condition~\ref{defRateTesting1}, one takes an arbitrary $\delta > 0$ and proposes a test statistic $\Deltahat$ satisfying
\begin{equation*}
         \Pb_0(\Deltahat = 1) + \sup_{\sol \in \setAlt(C\phi_{\epsilon,\sigma})} \Pb_1(\Deltahat = 0) \leq \delta
\end{equation*}
for all $C$ sufficiently large.
The proof of the lower bound~\ref{defRateTesting2} is similar to the one of lower bounds for the estimation problem and is mainly based on the auxiliary Lemma~\ref{lemLowerTesting} in the appendix.
This two-step program is realized for signal detection ($\solcirc = 0$) and goodness-of-fit testing ($\solcirc \neq 0$) separately in Section~\ref{secTesting}. \section{Minimax upper bound}\label{secUpper}
Our first theorem provides an upper risk bound for the estimator $\qhat_k$ for arbitrary $k \in \Nast$.
\begin{theorem}\label{thmUpper}
	Let Assumption~\ref{assSeq} hold.
	Then, for any $k \in \Nast$, the estimator $\qhat_k$ defined in~\eqref{eqDefEstSeries} satisfies, for any $\sol, \solcirc \in \csol$, the risk bound
\begin{align*}
	\sup_{\ev \in \cev} \Eb [(\qhat_k - \qfunc(\sol))]^2 &\leq C(d)\epsilon^4 \sum_{j=1}^{k} \omega_j^4\alpha_j^{-4} + C(d)\sigma^4 \sum_{j=1}^k  \omega_j^4\alpha_j^{-4}(\solcirc_j)^4 \\
	&\hspace{-1em}+C(d)\epsilon^2 \sum_{j=1}^k \omega_j^4 \alpha_j^{-2}(\sol_j - \solcirc_j)^2 + C(d) \sigma^2 \sum_{j=1}^k \omega_j^4 \alpha_j^{-2}(\solcirc_j)^2 (\sol_j - \solcirc_j)^2\\
	&\hspace{-1em}+C(d,L)\sigma^4 \sum_{j=1}^{k}  \omega_j^4\alpha_j^{-4}\gamma_j^{-2}(\sol_j - \solcirc_j)^2 +C(d,L)\sigma^2 \sum_{j=1}^{k}  \omega_j^4\alpha_j^{-2}\gamma_j^{-2} (\sol_j - \solcirc_j)^2  \\
	&\hspace{-1em}+\frac{\omega_k^4}{\gamma_k^4} \bigg( \sum_{j>k}\gamma_j^2 (\sol_j - \solcirc_j)^2 \bigg)^2.
	\end{align*}
	Consequently, for $\solcirc \in \csol$,
\begin{align*}
	\sup_{\sol \in \csol} \sup_{\ev \in \cev} \Eb [(\qhat_k - \qfunc(\sol))]^2  &\leq C(d) \epsilon^4 \sum_{j=1}^{k} \omega_j^4\alpha_j^{-4} + C(d,L) \epsilon^2 \max_{j \in \llbracket 1,k\rrbracket} \frac{\omega_j^4}{\alpha_j^2 \gamma_j^2} +  C(L) \, \frac{\omega_k^4}{\gamma_k^4}\\
	&\hspace{1em}+ C(d,L)\sigma^2 \max_{j \in \llbracket 1,k\rrbracket} \frac{\omega_j^4}{\alpha_j^2 \gamma_j^4} + C(d,L)\sigma^4 \max_{j \in \llbracket 1,k\rrbracket}  \frac{\omega_j^4}{\alpha_j^4 \gamma_j^4}.
	\end{align*}
\end{theorem}

\begin{proof}
	We consider the decomposition $\qhat_k - \qfunc(\sol) = \Tc_{k1} + \Tc_{k2} + \Tc_{k3} + \Tc_{k4}$ where
	\begin{align*}
		&\Tc_{k1} = \sum_{j=1}^{k} \omega_j^2 \, \frac{U_j}{V_j} \1_{\Omega_j} - \sum_{j=1}^k \omega_j^2 \, \frac{\lambda_j^2 (\sol_j - \solcirc_j)^2 }{V_j} \1_{\Omega_j},\\
		&\Tc_{k2} = \sum_{j=1}^k \omega_j^2 \, \frac{\lambda_j^2 (\sol_j - \solcirc_j)^2 }{V_j} \1_{\Omega_j} - \sum_{j=1}^k \omega_j^2 (\sol_j - \solcirc_j)^2 \1_{\Omega_j},\\
		&\Tc_{k3} = - \sum_{j=1}^k \omega_j^2 (\sol_j - \solcirc_j)^2 \1_{\Omega_j^\complement},\\
		&\Tc_{k4} = \sum_{j>k} \omega_j^2 (\sol_j - \solcirc_j)^2.
	\end{align*}
	Thus $\Eb [(\qhat_k - \qfunc(\sol))^2] \leq 4 \sum_{i=1}^{4} \Eb \Tc_{ki}^2$, and the rest of the proof consists in finding appropriate upper bounds  for the terms $\Eb \Tc_{ki}^2$, $i \in \llbracket 1,4\rrbracket$, which are derived in Appendix~\ref{appUpper}.
\end{proof}

The upper bound proved in Theorem~\ref{thmUpper} consists of terms that are non-decreasing in $k$, and the term $\omega_k^4\gamma_k^{-4}$ which is non-increasing in $k$.
Putting
\begin{align}
	&\kepsilon = \argmin_{k \in \Nast} \max \left\lbrace \epsilon^4 \sum_{j=1}^{k} \omega_j^4\alpha_j^{-4}, \epsilon^2 \max_{j \in \llbracket 1,k\rrbracket} \frac{\omega_j^4}{\alpha_j^2 \gamma_j^2}, \frac{\omega_k^4}{\gamma_k^4} \right\rbrace \qquad \text{and}\label{eqDef:kepsilon}\\
	&\ksigma = \argmin_{k \in \Nast} \max \left\lbrace \sigma^2 \max_{j \in \llbracket 1,k\rrbracket} \frac{\omega_j^4}{\alpha_j^2 \gamma_j^4}, \sigma^4 \max_{j \in \llbracket 1,k\rrbracket} \frac{\omega_j^4}{\alpha_j^4 \gamma_j^4}, \frac{\omega_k^4}{\gamma_k^4} \right\rbrace,\label{eqDef:ksigma}
\end{align}
the quantity $\kepsilon$ yields the best balance between the squared bias and the variance terms in $\epsilon$, and analogously $\ksigma$ the best balance between squared bias and variance terms in terms of $\sigma$.
Thus, the following corollary holds.

\begin{corollary}\label{corDiagOpt}
	Under the assumptions of Theorem~\ref{thmUpper}, $\kstar \defeq \kepsilon \wedge \ksigma$ with $\kepsilon$, $\ksigma$ as in \eqref{eqDef:kepsilon} and \eqref{eqDef:ksigma} provides the optimal choice of $k$ in Theorem~\ref{thmUpper}, and it holds
	\begin{align*}
		\sup_{\sol \in \csol} \sup_{\ev \in \cev} \Eb [(\qhat_{\kstar} - \qfunc(\sol))]^2  &\lesssim \epsilon^4 \sum_{j=1}^{\kstar} \omega_j^4\alpha_j^{-4} + \epsilon^2 \max_{j \in \llbracket 1,\kstar \rrbracket} \frac{\omega_j^4}{\alpha_j^2 \gamma_j^2} +  \frac{\omega_{\kstar}^4}{\gamma_{\kstar}^4}\\
		&+ \sigma^2 \max_{j \in \llbracket 1,\kstar \rrbracket} \frac{\omega_j^4}{\alpha_j^2 \gamma_j^4} + \sigma^4 \max_{j \in \llbracket 1,\kstar \rrbracket} \frac{\omega_j^4}{\alpha_j^4 \gamma_j^4}
	\end{align*}
	where the numerical constant in $\lesssim$ depends on $d$ and $L$.
\end{corollary}
It is remarkable that for the estimation of quadratic functionals the optimal truncation parameter $\kstar$ depends on both $\epsilon$ and $\sigma$ whereas the optimal truncation parameter for the estimation of $\sol$ itself under $\ell^2$-loss can be chosen in dependence on $\epsilon$ only (see~\cite{johannes2013partially}, Theorem~2.5).
It is not difficult to obtain the rates of convergence for the specific choices of $\gamma$ and $\alpha$ introduced in Subsection~\ref{subsMinimax} (see Appendix~\ref{appCalcRates} for some detailed calculations).
These rates are summarized in Table~\ref{tableRatesQuadFun} for the special case that $\omega_j \equiv 1$.
Note that in all illustrations the rate in $\sigma$ is at least as fast as the one in $\epsilon$, a fact that can in general be seen from the abstract rates in Corollary~\ref{corDiagOpt}.
In some examples, the rate in $\sigma$ is even strictly faster than the one in $\epsilon$.
For instance in the case where all the smoothness assumptions are polynomial, one has, depending on the actual values of $p$ and $a$, to distinguish between three 'zones' of rates:
\begin{enumerate}
 \item if $2p \leq a$, then both rates are non-parametric and the overall rate is $\epsilon^{16p/(4a+4p+1)} \vee \sigma^{4p/a}$,
 \item if $2p \geq a$ but $p \leq a + 1/4$, then the rate in $\sigma$ is the parametric rate but with respect to $\epsilon$ we are in the non-parametric regime, and the overall rate is $\epsilon^{16p/(4a+4p+1)} \vee \sigma^2$,
 \item if $p \geq a + 1/4$, then we are in the parametric regime with respect to both noise levels and the rate is $\epsilon^2 \vee \sigma^2=(\epsilon \vee \sigma)^2$.
\end{enumerate}
This behaviour generalizes the classical elbow effect which is well known in terms of the noise level $\epsilon$.

{\centering
\begin{table}
\caption{Optimal rates of convergence for the estimation of quadratic functionals in case that $\omega_j \equiv 1$. Upper bounds are proved in Section~\ref{secUpper}, lower bounds in Section~\ref{secLower}.}\label{tableRatesQuadFun}
\renewcommand{\arraystretch}{1.8}
\footnotesize
\begin{tabularx}{14cm}{c!{\color{gray!25}\vrule}cc}
     & Sobolev class \color{Ivory4}(\boldmath$\gamma_j = j^p$) & Analytic class \color{Ivory4}(\boldmath$\gamma_j  \asymp e^{pj}$) \\
    \arrayrulecolor{gray!25}\hline
    Mildly ill-posed \boldmath(\color{Ivory4}$\alpha_j = j^{-a}$) & $\epsilon^{16p/(4a+4p+1)} \vee \epsilon^2 \vee \sigma^{4p/a} \vee \sigma^2$ & $\epsilon^2 \vee \sigma^2$\\    
    Severely ill-posed \color{Ivory4}(\boldmath$\alpha_j \asymp e^{-aj}$) & $\vert \log \epsilon \vert^{-4p} \vee \vert \log \sigma \vert^{-4p}$ & $\epsilon^{4p/(p+a)} \vee \epsilon^2 \vee \sigma^{4p/a} \vee \sigma^2$
\end{tabularx}
\normalsize
\end{table}}

\begin{remark}\label{remTSEs}
  Let us mention that, using estimates similar to the ones used in the proof of Theorem~\ref{thmUpper}, it would be possible to show that the estimator
  \begin{align*}
       \qhat_k &= \sum_{j=1}^k \frac{X_j^2 - \epsilon^2}{Y_j^2 - \sigma^2} \1_{ \{ Y_j^2 \geq 2 \sigma^2 \} } - 2\sum_{j=1}^k \solcirc_j\frac{X_j}{Y_j} \1_{ \{ Y_j^2 \geq 2 \sigma^2 \} } + \sum_{j=1}^k (\solcirc_j)^2
  \end{align*}
  attains the optimal rate of convergence provided that the truncation parameter is suitably chosen.
  Note that this estimator does not depend on the availability of two independent samples of the noisy eigenvalues.
  A theoretical guarantee similar to the one of Theorem~\ref{thmUpper} could, for this alternative estimator, be obtained under an even weaker assumption than Gaussianity (for instance, sub-Gaussianity) because no sample splitting is necessary for the definition of the estimator.
  However, we stick to the estimator defined in~\eqref{eqDefEstSeries} since it provides a representation of the risk bound that is more convenient for our application to testing.
  More precisely, several terms in the risk bound contain the expression $(\sol_j - \solcirc_j)^2$ which vanishes when $\sol = \solcirc$, and precisely this property is exploited when controlling the type I error of our test procedures.
\end{remark} \section{Minimax lower bounds}\label{secLower}

In this section, we derive lower bounds on the minimax risk in the sense of \eqref{eqMinimaxRiskLower}.
In order to cleanse the notation, we restrict ourselves without loss of generality to the case $\solcirc=0$ (the proofs in the general case follow easily by adapting this proof to the case $\solcirc \neq 0$).
Note that the assumptions imposed in addition to Assumption~\ref{assSeq} in this section are satisfied by all our illustrating examples.
Thus, the results of this section imply the optimality of the rates in Table~\ref{tableRatesQuadFun}.

\subsection{Lower bounds in terms of $\epsilon$}

The following theorem provides a lower bound for the case that the rate with respect to $\epsilon$ is determined by the term $\epsilon^4 \sum_{j=1}^\kappa \omega_j^4 \alpha_j^{-4}$ (non-parametric regime) where
\begin{equation}\label{eqKappaEpsNP}
  \kappa = \argmin_{k \in \Nast} \max \bigg\{ \epsilon^4 \sum_{j=1}^{k} \omega_j^4\alpha_j^{-4}, \frac{\omega_k^4}{\gamma_k^4} \bigg\}.
\end{equation}
\begin{theorem}\label{thmLowerEps1}
Let Assumption~\ref{assSeq} hold true, and let $\kappa$ be defined as in~\eqref{eqKappaEpsNP}.
If
\begin{equation*}
 \epsilon^4 \sum_{j=1}^{\kappa} \omega_j^4\alpha_j^{-4} \asymp_\nu \omega_\kappa^4\gamma_\kappa^{-4}
\end{equation*}
for some $\nu \geq 1$, then
\begin{equation*}
        \inf_\qtilde \sup_{\sol \in \csol} \sup_{\ev \in \cev} \Eb [ ( \qtilde - \qfunc(\sol) )^2 ] \gtrsim \epsilon^4 \sum_{j=1}^{\kappa} \omega_j^4\alpha_j^{-4}
\end{equation*}
where the infimum is taken over all estimators $\qtilde$ based on the observations~\eqref{eqObsX} and~\eqref{eqObsY}.
\end{theorem}

The next theorem considers the case that the rate in $\epsilon$ is determined by balancing the terms $\epsilon^2 \max_{j \in \llbracket 1,k \rrbracket} \omega_j^4/(\alpha_j\gamma_j)^2$ and the squared bias $\omega_k^4 \gamma_k^{-4}$ (which might result in the parametric rate $\epsilon^2$ as the lower bound).

\begin{theorem}\label{thm:lower:epsilon:2}
Let Assumption~\ref{assSeq} hold true.
\begin{enumerate}[(a)]
    \item\label{lower:epsilon:2:a} 
    Set
    \begin{equation*}
      \kappa = \argmin_{k \in \Nast} \max \bigg\{ \epsilon^2 \frac{\omega_k^4}{\alpha_k^{2}\gamma_k^{2}}, \frac{\omega_k^4}{\gamma_k^4} \bigg\},
    \end{equation*}
    and assume that $\epsilon^2 \alpha_\kappa^{-2} \gamma_\kappa^{-2} \asymp_\nu \gamma_\kappa^{-4}$.
    Then
    \begin{equation*}
            \inf_\qtilde \sup_{\theta \in \csol} \sup_{\lambda \in \cev} \Eb [ ( \qtilde - \qfunc(\theta) )^2 ] \gtrsim \epsilon^2 \frac{\omega_\kappa^4}{\alpha_\kappa^2\gamma_\kappa^2}
    \end{equation*}
    where the infimum is taken over all estimators $\qtilde$.
    \item\label{thm:lower:epsilon:2:b}
    It holds
    \begin{equation*}
            \inf_\qtilde \sup_{\theta \in \csol} \sup_{\lambda \in \cev} \Eb [ ( \qtilde - \qfunc(\theta) )^2 ] \geq \min \left\lbrace \frac{L^4}{16}, \frac{1}{4d^4} \right\rbrace  \cdot \epsilon^2
    \end{equation*}
    where the infimum is taken over all estimators $\qtilde$.
\end{enumerate}
\end{theorem}

For the illustrative examples of rates in Table~\ref{tableRatesQuadFun}, two different scenarios can occur.
In the first case, the sequence $\omega^ 4\alpha^{-2}\gamma^{-2}$ is eventually monotone and tends to infinity.
Then, Statement~\ref{lower:epsilon:2:a} of Theorem~\ref{thm:lower:epsilon:2} applies with the additional assumption of this statement being fulfilled for all our examples.
In the second case, the sequence $\omega^ 4\alpha^{-2}\gamma^{-2}$ is bounded from above, and we apply Statement~\ref{thm:lower:epsilon:2:b} in order to get the parametric rate $\epsilon^2$ as a lower bound.
The proof of the parametric rate $\epsilon^2$ in \ref{thm:lower:epsilon:2:b} given in Appendix~\ref{appLowerEps2} might be of independent interest, since it provides an alternative to the classical approach given in~\cite{fan1991estimation} (see also~\cite{fan1992minimax}) who reduces the proof of the lower bound $\epsilon^2$ to the estimation of a quadratic functional in the normal bounded mean model.

\subsection{Lower bounds in terms of $\sigma$}

We now tackle the problem of finding lower bounds with respect to the noise level $\sigma$.
\begin{theorem}\label{thmLowerSigma}
	Let Assumption~\ref{assSeq} hold true.
	\begin{enumerate}[(a)]
		\item\label{thmLowerSigma:a}
		Set
		\begin{equation*}
		  \kappa = \argmin_{k \in \Nast} \frac{\omega_k^4}{\gamma_k^4} \max \{ \sigma^2 \alpha_k^{-2},1 \},
		\end{equation*}
		and assume $\sigma^2 \alpha_\kappa^{-2} \asymp_\nu 1$ for some $\nu \geq 1$ independent of $\sigma$.
		Then,
\begin{equation*}
	\inf_\qtilde \sup_{\sol \in \csol} \sup_{\ev \in \cev} \Eb [ ( \qtilde - \qfunc(\theta) )^2 ] \gtrsim \min_{k \in \Nast} \omega_k^4 \gamma_k^{-4} \max \{ \sigma^2\alpha_k^{-2},1  \}
\end{equation*}
where the infimum is taken over all estimators $\qtilde$ of $\qfunc(\sol)$.
		\item\label{thmLowerSigma:b} 
		It holds
		\begin{equation*}
			\inf_\qtilde \sup_{\sol \in \csol} \sup_{\ev \in \cev} \Eb [ ( \qtilde - \qfunc(\theta) )^2 ] \gtrsim \sigma^2
		\end{equation*}
		where the infimum is taken over all estimators $\qtilde$ of $\qfunc(\sol)$.
	\end{enumerate}
\end{theorem}

As for Theorem~\ref{thm:lower:epsilon:2}, essentially two scenarios can occur.
In the first case, the sequence $\omega^4 \alpha^{-2} \gamma^{-4}$ is eventually monotone and tends to infinity.
Then, Statement~\ref{thmLowerSigma:a} of Theorem~\ref{thmLowerSigma} gives the desired lower bound and the additional assumption in Statement~\ref{thmLowerSigma:a} is fulfilled by all our examples.
In the second case, the sequence $\omega^ 4\alpha^{-2}\gamma^{-4}$ is bounded from above, and we apply Statement~\ref{thmLowerSigma:b} in order to get the parametric rate $\sigma^2$ as a lower bound. \section{Application to testing problems}\label{secTesting}
As announced in the introduction we apply the theory developed in the previous sections to signal detection and goodness-of-fit testing separately.

\subsection{Signal detection}\label{ssSignal} We start by considering the signal detection problem of testing
\begin{equation*}
	\Hc_0\colon \sol = 0 \qquad \text{against} \qquad \Hc_1\colon \sol \in \setAlt(r)
\end{equation*}
for $r > 0$ where $\setAlt(r) = \csol \cap \{ \sol \colon \Vert \sol \Vert_2 \geq r \}$.
It turns out that for this problem it is sufficient to consider observations~\eqref{eqObsX}, and to construct a test statistic which is based on an estimator of the quadratic functional $\qfuncsd(\soltilde) = \sum_{j=1}^{\infty} \alpha_j^{-2} \soltilde_j^2$ where $\soltilde = \lambda \sol$.
Note that the estimation of this quadratic functional from~\eqref{eqObsX} is not an inverse but a direct problem since, in terms of $\soltilde$, \eqref{eqObsX} reads
\begin{equation}\label{eqObsTildeX}
	X_j = \soltilde_j + \epsilon \xi_j.
\end{equation}
Moreover, the smoothness assumptions in the original model transfer to smoothness assumptions for $\soltilde$, namely that $\soltilde$ belongs to an ellipsoid with weight sequence $\widetilde{\gamma} = \gamma \alpha^{-1}$.
In addition, the weight sequence in the definition $\qfuncsd(\soltilde)$ is $\omega_j = \alpha_j^{-1}$ in this case.
The choice of the truncation value for our auxiliary estimator is slightly different from the optimal choice in Corollary~\ref{corDiagOpt}.
More precisely, we put
\begin{equation*}
 \ksd = \argmin_{k \in \Nast} \max \Bigg\{  \epsilon^4 \sum_{j=1}^k \alpha_j^{-4} ,\gamma_k^{-4} \Bigg\},
\end{equation*}
and define
\begin{equation}
	\qhat_{\ksd} = \sum_{j=1}^{\ksd} \alpha_j^{-2} (X_j^2 - \epsilon^2).\label{eqDefQhatsd}
\end{equation}
Now, in order to prove the upper bound for the minimax testing rate, introduce the test statistic 
\begin{equation*}
	\Deltahatsd = \1_{\{\qhatksd \geq \Ctilde \phi_\epsilon^2 \} } \quad \text{where} \quad \phi_\epsilon^2 = \epsilon^2 \sqrt{\sum_{j=1}^{\ksd} \alpha_j^{-4}},
\end{equation*}
and $\Ctilde$ is a numerical constant that has to be chosen appropriately, see Theorem~\ref{UpperSignalDetection} below.
The proof of the following Theorem~\ref{UpperSignalDetection} shows that the test statistic $\Deltahatsd$ satisfies property~\ref{defRateTesting1} in Definition~\ref{defRateTesting} for the rate $\phi_\epsilon^2$.

\begin{theorem}\label{UpperSignalDetection}
	Let Assumption~\ref{assSeq} be satisfied, and assume that in addition $\gamma_{\ksd}^{-2} \leq \sqrt{\nu} \phi_\epsilon^2$ for some $\nu \geq 1$.
	Let $\delta \in (0,1)$ be fixed. Then, $\Rc(C\phi_\epsilon) \leq \delta$ for all sufficiently large $C$.
\end{theorem}

The next theorem provides the corresponding lower bound in the sense of Condition~\ref{defRateTesting2} from Definition~\ref{defRateTesting}.

\begin{theorem}\label{thmLowerSignalDetection}
	Let $\delta \in (0,1)$ be arbitrary.
	Let Assumption~\ref{assSeq} hold true, and assume in addition that $\epsilon^4 \sum_{j=1}^{\ksd} \alpha_j^{-4} \asymp_\nu \gamma_{\ksd}^{-4}$.
	Then, there exists $C_\ast > 0$ such that for all $0 < c < C_\ast$ the inequality $\Rc (c\phi_\epsilon) \geq 1-\delta$ holds.
\end{theorem}
	
Specializing the results of Theorems~\ref{UpperSignalDetection} and~\ref{thmLowerSignalDetection} with our standard illustrations, we obtain the minimax rates of testing for signal detection in all the considered cases.
These are summarized in Table~\ref{tableRatesSignal}.
	
{\centering
\begin{table}
\caption{Optimal minimax rates of testing for the signal detection problem under the assumptions of Theorems~\ref{UpperSignalDetection}  and \ref{thmLowerSignalDetection}.}\label{tableRatesSignal}
\renewcommand{\arraystretch}{1.8}
\footnotesize
\begin{tabularx}{14cm}{c!{\color{gray!25}\vrule}cc}
     & Sobolev class \color{Ivory4}(\boldmath$\gamma_j = j^p$) & Analytic class \color{Ivory4}(\boldmath$\gamma_j \asymp e^{pj}$) \\
    \arrayrulecolor{gray!25}\hline
    Mildly ill-posed \boldmath(\color{Ivory4}$\alpha_j = j^{-a}$) & $\epsilon^{\frac{8p}{4a+4p+1}}$ & $\epsilon^2 \vert \log \epsilon \vert^{2a+\frac{1}{2}}$ \\    
    Severely ill-posed \color{Ivory4}(\boldmath$\alpha_j \asymp e^{-aj}$) & $\vert \log \epsilon \vert^{-2p}$ & $\epsilon^{\frac{2p}{a+p}}$ 
\end{tabularx}
\normalsize
\end{table}}

\subsection{Goodness-of-fit testing}\label{ssGoF} In this subsection, we consider the goodness-of-fit testing problem given by testing
\begin{equation*}
	\Hc_0 \colon \sol = \solcirc \in \csol \qquad \text{against} \qquad \Hc_1 \colon \sol \in \csol, \sol-\solcirc \in \setAlt(r).
\end{equation*}
In contrast to the signal detection problem considered above, the minimax rate of testing will now depend also on the noise level $\sigma$.
In the sequel, we make the technical assumption that \emph{all} the components of $\solcirc$ are non-zero: if this is not the case, one applies the signal detection methodology from Subsection~\ref{ssSignal} to test the components $\solcirc_j$ where $\solcirc_j = 0$ and combines this approach with the results derived in the sequel.
The fact that $\solcirc_j$ is non-zero is explicitly exploited in the proof of Theorem~\ref{LowerGoF} where it guarantees that the constructed hypotheses concerning the solution $\sol$ are distinct.
We consider the estimator $\qhatkgof$ of the quadratic functional $\qfuncgof(\sol) = \sum_{j=1}^{\infty} (\sol_j - \solcirc_j)^2$ (that is, $\omega_j \equiv 1$ in terms of our general notation) defined through
\begin{equation*}
	\qhatkgof = \sum_{j=1}^{\kgof} \frac{U_j}{V_j} \1_{\Omega_j}
\end{equation*}
with $U_j$, $V_j$, $\Omega_j$ defined as in Subsection~\ref{subsTSE}, and $\kgof$ defined as
\begin{equation*}
 \kgof = \argmin_{k \in \Nast} \max \bigg\{ \epsilon^2 \sqrt{\sum_{j=1}^k \alpha_j^{-4}}, \sigma^2 \max_{j \in \llbracket 1, k\rrbracket} \alpha_j^{-2}\gamma_j^{-2}, \gamma_k^{-2} \bigg\}
\end{equation*}
(again the definition of the threshold $\kgof$ slightly differs from the one in Corollary~\ref{corDiagOpt}).
Let us introduce the test statistic
\begin{equation}\label{defDeltahatgof}
	\Deltahatgof = \1_{ \{ \qhatkgof \geq \Ctilde \phi_{\epsilon, \sigma}^2 \} } \quad \text{where} \quad \phi_{\epsilon, \sigma}^2 = \max \bigg\{  \epsilon^2 \sqrt{\sum_{j=1}^{\kgof} \alpha_j^{-4}}, \sigma^2 \max_{j \in \llbracket 1,\kgof\rrbracket} \alpha_j^{-2}\gamma_j^{-2} \bigg\}.
\end{equation}
The following theorem shows that this statistic satisfies the upper bound condition~\ref{defRateTesting2} for $\Ctilde$ suitably chosen.

\begin{theorem}\label{thmUpperGoF}
 Let Assumption~\ref{assSeq} be satisfied, and assume that in addition $\gamma_{\kgof}^{-2} \leq \sqrt{\nu} \phi_{\epsilon, \sigma}^2$ for some $\nu \geq 1$.
Let $\delta \in (0,1)$ be fixed. Then, we have $\Rc(C\phi_\epsilon) \leq \delta$ for all sufficiently large $C$.
\end{theorem}

\begin{remark}\label{remLevelAlphaTests}
Note that, given $\alpha, \beta \in (0,1]$, following the same arguments as in the proof of Theorem~\ref{thmUpperGoF}, we could tune the numerical constant $\Ctilde$ in the definition of the test statistic such that $\Pb_0(\Deltahatgof = 1) \leq \alpha$ and $\Pb_\sol(\Deltahatgof = 0) \leq \beta$ for all $\sol \in \setAlt(C\phi_{\epsilon, \sigma})$ with $C$ sufficiently large.
This shows that the order of the separation rate in the sense of \cite{marteau2017minimax} is $\phi_{\epsilon, \sigma}^2$ (this rate was only derived as a lower bound in~\cite{marteau2017minimax} whereas the upper bound in that paper contains an additional logarithmic factor; however the test statistic considered in~\cite{marteau2017minimax} is already adaptive with respect to the class $\cev$ in the sense that its definition does neither depend on $\alpha$ nor on $d$).
It might be of interest to find out if the extra logarithmic factors appearing in the rate of \cite{marteau2017minimax} are optimal in the sense that no adaptive testing procedure can do without these terms.
\end{remark}

We now prove the lower bound on the minimax rate of testing for goodness-of-fit testing.

\begin{theorem}\label{LowerGoF}
	Let $\delta \in (0,1)$ be arbitrary.
	Let Assumption~\ref{assSeq} hold true, and assume that $\phi_{\epsilon,\sigma}^4 \asymp_\nu \gamma_{\kgof}^{-4}$.
	Then, there exists $C_\ast > 0$ such that for all $0 < c < C_\ast$ the inequality $\Rc (c\phi_\epsilon) \geq 1-\delta$ holds.
\end{theorem}

Again, specializing the results of Theorems~\ref{thmUpperGoF} and~\ref{LowerGoF} with our standard illustrations, we obtain the minimax rates of testing for the goodness-of-fit testing problem for all the considered cases.
These are summarized in Table~\ref{tableGoF}.

{\centering
\begin{table}
\caption{Optimal minimax rates of testing for goodness-of-fit testing under the assumptions of Theorems~\ref{thmUpperGoF}  and \ref{LowerGoF}.}\label{tableGoF}
\renewcommand{\arraystretch}{1.8}
\footnotesize
\begin{tabularx}{14cm}{c!{\color{gray!25}\vrule}cc}
     & Sobolev class \color{Ivory4}(\boldmath$\gamma_j = j^p$) & Analytic class \color{Ivory4}(\boldmath$\gamma_j  \asymp e^{pj}$) \\
    \arrayrulecolor{gray!25}\hline
    Mildly ill-posed \boldmath(\color{Ivory4}$\alpha_j = j^{-a}$) & $\epsilon^{\frac{8p}{4a+4p+1}} \vee \sigma^{2} \vee \sigma^{\frac{2p}{a}}$ & $\epsilon^2 \vert \log \epsilon \vert^{2a+\frac{1}{2}} \vee \sigma^2$ \\    
    Severely ill-posed \color{Ivory4}(\boldmath$\alpha_j \asymp e^{-aj}$) & $\vert \log \epsilon \vert^{-2p} \vee \vert \log \sigma \vert^{-2p}$ & $\epsilon^{\frac{2p}{a+p}} \vee \sigma^{2} \vee \sigma^{\frac{2p}{a}}$ 
\end{tabularx}
\normalsize
\end{table}} \section{Conclusion and open questions}\label{secDiscussion}

We have considered the minimax optimal estimation of quadratic functionals in the Gaussian sequence model given by \eqref{eqObsX} and \eqref{eqObsY}, and applied our theoretical findings to testing problems.
In particular, we have derived the minimax rates of estimation and minimax rates of testing under mild assumptions that allow us to deal with the classical examples from the literature.
A next step for future research might be to transfer the methodology developed in this paper to deconvolution models with unknown error distribution~\cite{comte2011data,johannes2009deconvolution}.
Apart from that, the following problems have not been dealt with in this paper and might be worth of being more closely investigated:
\begin{itemize}
	\item The optimal estimator of the quadratic functionals is not completely data-driven, and the definition of an adaptive selection rule for the truncation parameter that satisfies some theoretical guarantees is necessary.
	\item Equally, the problem of adaptive testing has not been discussed. In particular, can standard techniques for adaptive testing in inverse problems as developed in~\cite{butucea2009adaptive} (see also~\cite{lacour2014goodness}) be transferred to the model with partially unknown operator, and what is the price one has to pay for adaption?  
\item The general matrix case given by observations
	\begin{equation*}
	  X = \Op \sol + \epsilon \xi \quad \text{and} \quad Y = \Op + \sigma \Xi
	\end{equation*}
	is still open. Note that results for this model might be of interest since it is related to inverse problems like non-parametric instrumental variable regression or functional linear regression where non-diagonal matrices appear in a natural manner.
	\item Finally, considering inverse problems with sparsity constraints as in~\cite{collier2017minimax} might be of interest.
\end{itemize} 
\appendix

\section{General tools for lower bounds}\label{appGenToolsLowerBounds}

\subsection{Reduction to comparison with a fuzzy hypothesis}

For a probability measure $\mu$ on $\csol$, we put $\Pb_\mu^X = \int_{\csol} \Pb_\theta^X \mu(\dd \theta)$.
The following lemma reduces the problem of establishing a minimax lower bound on the class $\csol$ to the problem of testing $\Pb_0^X$ ($\mu = \delta_0$) against some $\Pb_\mu^X$ with $\mu \neq \delta_0$.
It is a special case of Theorem~2.15 in~\cite{tsybakov2009introduction} (the formulation is mainly borrowed from \cite{collier2017minimax}, see Lemma~2 therein).

\begin{lemma}\label{lemLowerHypercube}
	Let $\Theta$ be a subset of $\ell^2(\Nast)$ containing $0$. Assume that there exists a probability measure $\mu$ on $\Theta$ and numbers $\psi > 0$, $\beta > 0$ such that $\qfunc(\sol) = 2 \psi$ for all $\sol \in \supp(\mu)$ and $\chi^2(\Pb_\mu^X, \Pb_0^X) \leq \beta$.
	Then,
	\begin{equation*}
		\inf_\qtilde \sup_{\sol \in \csol} \sup_{\ev \in \cev} \Pb_{(\sol,\ev)}( \vert \qtilde - \qfunc(\sol) \vert \geq \psi) \geq \frac{1}{4} \exp(-\beta)
	\end{equation*} 
	where the infimum is taken over all estimators $\qtilde$.
\end{lemma}

\subsection{Reduction to two hypotheses}\label{ssecReduction2}

For the proofs of Theorems~\ref{thm:lower:epsilon:2} and~\ref{thmLowerSigma} we will construct hypotheses $(\sol^\tau,\ev^\tau) \in \csol \times \cev$ for $\tau \in \setpm$ such that the Kullback-Leibler distance between the resulting distributions $\Pb_1$ and $\Pb_{-1}$ of the tuple $(X,Y)$ is bounded by $1$.
This implies $\rho(\Pb_1, \Pb_{-1}) \geq 1/2$ for the Hellinger affinity being defined as $\rho(\Pb_1, \Pb_{-1}) = \sqrt{\dd \Pb_{1}\dd \Pb_{-1}}$.
Putting $\qfunc_\tau = \qfunc(\sol^\tau)$ for $\tau \in \{ \pm 1 \}$ we can conclude from
\begin{align*}
	\frac{1}{2} &\leq \int \frac{\abs{\qtilde - \qfunc_1}}{\abs{\qfunc_1 - \qfunc_{-1}}} \sqrt{\dd\Pb_{1}\dd\Pb_{-1}} + \int \frac{\abs{\qtilde - \qfunc_{-1}}}{\abs{\qfunc_1 - \qfunc_{-1}}} \sqrt{\dd \Pb_{1}\dd\Pb_{-1}}\\
	&\leq \bigg( \int \bigg( \frac{\qtilde - \qfunc_1}{\qfunc_1 - \qfunc_{-1}} \bigg)^2 \dd \Pb_1 \bigg)^{1/2} + \bigg( \int \bigg( \frac{\qtilde - \qfunc_{-1}}{\qfunc_1 - \qfunc_{-1}} \bigg)^2 \dd\Pb_{-1} \bigg)^{1/2}
\end{align*}
by using the elementary estimate $(a+b)^2 \leq 2a^2 + 2b^2$ that
\begin{align*}
	\frac{1}{8} \ (\qfunc_1 - \qfunc_{-1})^2 \leq \Eb_1 [( \qtilde - \qfunc_1 )^2] + \Eb_{-1} [( \qtilde - \qfunc_{-1} )^2].
\end{align*}
This last estimate in turn yields
\begin{align}
	\sup_{\sol \in \csol} \sup_{\ev \in \cev} \Eb [(\qtilde - \qfunc(\sol))^2] \geq \frac{1}{2} \sum_{\tau \in \setpm} \Eb_\tau [(\qtilde - \qfunc_\tau)^2] \geq \frac{1}{16} (\qfunc_1 - \qfunc_{-1})^2 \label{eqRedSch}
\end{align}
which establishes the quantity $\frac{1}{16} (\qfunc_1 - \qfunc_{-1})^2$ as a lower bound on the minimax rate.

\subsection{Reduction argument for lower bounds of testing}

\begin{lemma}\label{lemLowerTesting}
	Let $\mu$ be a probability measure on $\setAlt$.
	Then, the following statements hold true:
	\begin{enumerate}[(i)]
		\item $\inf_\Delta \left\lbrace \Pb_0(\Delta = 1) + \sup_{\sol \in \setAlt} \Pb_\sol(\Delta = 0) \right\rbrace \geq 1 - \sqrt{\chi^2(\Pb_\mu, \Pb_0)}$,\label{lemLowerTesting1}
		\item $\inf_\Delta \left\lbrace \Pb_0(\Delta = 1) + \sup_{\sol \in \setAlt} \Pb_\sol(\Delta = 0) \right\rbrace \geq 1 - \sqrt{\KL(\Pb_\mu, \Pb_0)/2}$.\label{lemLowerTesting2}
	\end{enumerate}
	In both statements, the infimum is taken over all $\{ 0,1 \}$-valued statistics.
\end{lemma}

\begin{proof}
	For any $\{ 0,1 \}$-valued statistic $\Delta$,
	\begin{align*}
		\Pb_0(\Delta = 1) + \sup_{\sol \in \setAlt} \Pb_\sol(\Delta = 0) &\geq \Pb_0(\Delta = 1) + \int_{\setAlt} \Pb_\sol (\Delta = 0) \mu(\dd \sol)\\
		&= \Pb_0(\Delta = 1) + \Pb_\mu (\Delta = 0)\\
		&\geq 1- \TV(\Pb_\mu, \Pb_0).
	\end{align*}
	Therefrom, Statement~\ref{lemLowerTesting1} follows using Equation~(2.27) in~\cite{tsybakov2009introduction}, and Statement~\ref{lemLowerTesting2} by the first Pinsker inequality (see~\cite{tsybakov2009introduction}, Lemma~2.5).
\end{proof} \section{Upper bounds for the terms $\Eb \Tc_{ki}^2$ in the proof of Theorem~\ref{thmUpper}}\label{appUpper}

\paragraph{\emph{Upper bound for $\Eb \Tc_{k1}^2$}} By independence of $U_j$ and $\1_{\Omega_j}/V_j$ and $\Eb[U_j - \lambda_j^2 (\sol_j - \solcirc_j)^2]=0$, it holds
	\begin{align*}
		\Eb \left[ \left( \sum_{j=1}^{k} \omega_j^2 \, \frac{U_j - \lambda_j^2 (\sol_j - \solcirc_j)^2}{V_j} \1_{\Omega_j} \right)^2 \right]  &= \var \left( \sum_{j=1}^{k} \omega_j^2 \, \frac{U_j - \lambda_j^2 (\sol_j - \solcirc_j)^2}{V_j} \1_{\Omega_j} \right)\\
		&=\sum_{j=1}^{k} \omega_j^4 \var \left( \frac{U_j - \lambda_j^2 (\sol_j - \solcirc_j)^2}{V_j} \1_{\Omega_j} \right). 
	\end{align*}
	Set $Z_1 = (U_j - \lambda_j^2 (\sol_j - \solcirc_j)^2)/\lambda_j^2$ and $Z_2 = \lambda_j^2/V_j \cdot \1_{\Omega_j}$.
	Note that $Z_1$ and $Z_2$ are independent, and since $\Eb Z_1 = 0$, the identity $\var(Z_1Z_2) = \var(Z_1)\var(Z_2) + \var(Z_1)(\Eb Z_2)^2 + \var(Z_2)(\Eb Z_1)^2$ reduces to $\var(Z_1Z_2) = \var(Z_1) \Eb (Z_2^2)$.
	Hence,
	\begin{align*}
		\var \left( \frac{U_j - \lambda_j^2 (\sol_j - \solcirc_j)^2}{V_j} \1_{\Omega_j} \right) &= \var(Z_1) \Eb [Z_2^2] \\
		&= \var \left( \frac{U_j - \lambda_j^2(\sol_j - \solcirc_j)^2}{\lambda_j^2} \right) \cdot \Eb \left[ \frac{\lambda_j^4}{V_j^2} \1_{\Omega_j} \right]\\
		&\leq 168\lambda_j^{-4} \var (U_j)
	\end{align*}
	where we have used Statement~\ref{UpperAux1} from Proposition~\ref{UpperAux}.
	Now, since $\var(U_j) = 2(\epsilon^2 + 2 \sigma^2 (\solcirc_j)^2)^2 + 4(\epsilon^2 + 2 \sigma^2 (\solcirc_j)^2) \lambda_j^2 (\sol_j - \solcirc_j)^2$, we obtain using $(a+b)^2 \leq 2a^2+2b^2$ that
	\begin{align*}
		\var \left( \frac{U_j - \lambda_j^2 (\sol_j - \solcirc_j)^2}{V_j} \1_{\Omega_j} \right)&\\
		&\hspace{-3em}\leq 672 \epsilon^4 \lambda_j^{-4} + 2688 \sigma^4 (\solcirc_j)^4 \lambda_j^{-4} + 672 (\epsilon^2 + 2 \sigma^2 (\solcirc_j)^2) \lambda_j^{-2} (\sol_j - \solcirc_j)^2.
	\end{align*}
	Now summation over all indices $j \in \llbracket 1, k \rrbracket$ yields
	\begin{align*}
		\Eb \Tc_{k1}^2 &\leq 672 \epsilon^4 \sum_{j=1}^k \omega_j^4 \lambda_j^{-4} + 2688 \sigma^4 \sum_{j=1}^k \omega_j^4\lambda_j^{-4} (\solcirc_j)^4 \\
		&\hspace{1em}+ 672\epsilon^2 \sum_{j=1}^k \omega_j^4 \lambda_j^{-2} (\sol_j - \solcirc_j)^2 + 1344 \sigma^2 \sum_{j=1}^k \omega_j^4 \lambda_j^{-2} (\solcirc_j)^2 (\sol_j - \solcirc_j)^2\\
		&\leq 672d^4 \epsilon^4 \sum_{j=1}^k \omega_j^4 \alpha_j^{-4} + 2688d^4 \sigma^4 \sum_{j=1}^k \omega_j^4\alpha_j^{-4} (\solcirc_j)^4 \\
		&\hspace{1em}+ 672d^2\epsilon^2 \sum_{j=1}^k \omega_j^4 \alpha_j^{-2} (\sol_j - \solcirc_j)^2 + 1344d^2 \sigma^2 \sum_{j=1}^k \omega_j^4 \alpha_j^{-2} (\solcirc_j)^2 (\sol_j - \solcirc_j)^2.
	\end{align*}

	\paragraph{\emph{Upper bound for $\Eb \Tc_{k2}^2$}} Using the Cauchy-Schwarz inequality, it holds
	\begin{align*}
	\Eb \Tc_{2k}^2 &= \Eb \left[ \left( \sum_{j=1}^{k} \omega_j^2\lambda_j^2 (\sol_j- \solcirc_j)^2 \left( \frac{1}{V_j} - \frac{1}{\lambda_j^2} \right) \1_{ \Omega_j } \right)^2 \right]\\
	&\leq \Eb \left[ \left( \sum_{j=1}^{k} \gamma_j^2(\sol_j - \solcirc_j)^2  \right) \left( \sum_{j=1}^{k} \omega_j^4\gamma_j^{-2} (\sol_j - \solcirc_j)^2  \left( \frac{\lambda_j^2}{V_j} -1 \right)^2 \1_{ \Omega_j } \right) \right]\\
	&\leq 4L^2 \sum_{j=1}^{k} \omega_j^4 (\sol_j - \solcirc_j)^2 \gamma_j^{-2} \Eb \left[ \left( \frac{\lambda_j^2}{V_j} - 1 \right)^2 \1_{\Omega_j} \right]\\
	&\leq C(d)L^2 \sum_{j=1}^{k} \omega_j^4\gamma_j^{-2} (\sigma^4 \alpha_j^{-4} + \sigma^2 \alpha_j^{-2}) (\sol_j - \solcirc_j)^2 
	\end{align*}
	where the last estimate is due to Statement~\ref{UpperAux2} from Proposition~\ref{UpperAux}.
	\paragraph{\emph{Upper bound for $\Eb \Tc_{k3}^2$}} 
	Again by the Cauchy-Schwarz inequality we have
	\begin{align*}
	\Eb \left[ \left( \sum_{j=1}^{k} \omega_j^2 (\sol_j-\solcirc_j)^2 \1_{ \Omega_j^\complement } \right)^2 \right] &\leq \Eb \left[ \left( \sum_{j=1}^{k} \gamma_j^2(\sol_j - \solcirc_j)^2  \right) \left( \sum_{j=1}^{k} \omega_j^4 \gamma_j^{-2}(\sol_j - \solcirc_j)^2  \1_{\Omega_j^\complement} \right) \right]\\
	&\leq 4L^2 \sum_{j=1}^{k} \omega_j^4\gamma_j^{-2} (\sol_j - \solcirc_j)^2  \Pb (\Omega_j^\complement).
	\end{align*}
	Bounding the probability of the event $\Omega_j^\complement$ by means of Statement~\ref{UpperAux3} in Proposition~\ref{UpperAux}, we conclude
	\begin{equation*}
	\Eb \Tc_{k3}^2 \leq 48d^2L^2 \sum_{j=1}^{k} \omega_j^4\gamma_j^{-2} (\sol_j - \solcirc_j)^2  \min \{ 1,\sigma^2 \alpha_j^{-2} \}.
	\end{equation*}
	
	\paragraph{\emph{Upper bound for $\Eb \Tc_{k4}^2$}} Note that $\Tc_{k4}^2$ is deterministic. Hence,
	\begin{equation*}
		\Eb \Tc_{k4}^2 = \Tc_{k4}^2 = \left( \sum_{j > k} \omega_j^2 (\sol_j - \solcirc_j)^2 \right)^2 \leq \frac{\omega_k^4}{\gamma_k^4} \left( \sum_{j>k} \gamma_j^2 (\sol_j - \solcirc_j)^2 \right)^2 \leq 16L^4 \cdot \frac{\omega_k^4}{\gamma_k^4}.
	\end{equation*} \section{Auxiliary results for the Proof of Theorem~\ref{thmUpper}}\label{app:aux:upper}

\begin{proposition}\label{UpperAux}
	With the notations introduced in the main part of the paper the following assertions hold true:
	\begin{enumerate}[(i)]
		\item\label{UpperAux1} $\Eb \left[ \frac{\lambda_j^4}{V_j^2}  \1_{\Omega_j} \right] \leq 168$,
		\item\label{UpperAux2} $\Eb \left[ \left( \frac{\lambda_j^2}{V_j}  - 1 \right)^2 \1_{\Omega_j} \right] \leq C(d) \sigma^4 \alpha_j^{-4} + C(d)\sigma^2 \alpha_j^{-2}$,
		\item\label{UpperAux3} $\Pb (\Omega_j^\complement) \leq 12d^2 \min (1,\sigma^2 \alpha_j^{-2})$.
	\end{enumerate}
\end{proposition}

\begin{proof}
	We begin the proof of \ref{UpperAux1} with the observation that, since the function $x \mapsto \frac{x}{x-2\sigma^2}$ is non-increasing on $[3\sigma^2, \infty)$, 
	\begin{equation}\label{eqPropDiagUpperAux}
		\frac{\Ypp_j^4}{V_j^2} \1_{\Omega_j}  = \left( \frac{\Ypp_j^2}{V_j} \right)^2 \1_{\Omega_j} = \left( \frac{\Ypp_j^2}{\Ypp_j^2 - 2\sigma^2} \right)^2 \1_{\Omega_j} \leq \left( \frac{3\sigma^2}{\sigma^2} \right)^2 \leq 9.
	\end{equation}
	Therefrom, using $(a+b)^4 \leq 8a^4 + 8b^4$
	\begin{align*}
		\Eb \left[ \frac{\lambda_j^4}{V_j^2} \1_{\Omega_j} \right] &\leq \Eb \left[ \frac{\lambda_j^4}{\Ypp_j^4} \cdot \frac{\Ypp_j^4}{V_j^2}  \1_{\Omega_j} \right] \leq 9 \Eb \left[ \frac{\lambda_j^4}{\Ypp_j^4} \1_{\Omega_j} \right]\\
		&\leq 9 \Eb \left[  \frac{(\lambda_j - \Ypp_j + \Ypp_j)^4}{\Ypp_j^4} \1_{\Omega_j} \right]\\
		&\leq 72 \Eb \left[ \frac{(\lambda_j -  \Ypp_j)^4}{9\sigma^4} \right] + 72 \leq 96 + 72 = 168. 
	\end{align*}
	In order to show \ref{UpperAux2}, introduce the event $\mho_j \defeq \left\lbrace \left| \frac{1}{\Ypp_j} - \frac{1}{\lambda_j} \right| \leq \frac{1}{2 \abs{\lambda_j}}  \right\rbrace$.
	Then, trivially,
	\begin{equation}\label{eqPropDiagUpperAux:2}
		\Eb \left[ \left( \frac{\lambda_j^2}{V_j}  - 1 \right)^2 \1_{\Omega_j} \right] = \Eb \left[ \left( \frac{\lambda_j^2}{V_j}  - 1 \right)^2 \1_{\Omega_j} (\1_{\mho_j} + \1_{\mho_j^\complement}) \right], 
	\end{equation}
	and we consider the summands with $\1_{\mho_j}$ and $\1_{\mho_j^\complement}$ separately.
	First, using~\eqref{eqPropDiagUpperAux},
	\begin{equation*}
		\Eb \left[ \left( \frac{\lambda_j^2}{V_j} - 1 \right)^2 \1_{\Omega_j} \1_{\mho_j} \right] = \Eb \left[ \frac{(\lambda_j^2 - V_j)^2}{\Ypp_j^4} \cdot \frac{\Ypp_j^4}{V_j^2} \cdot \1_{\Omega_j} \1_{\mho_j} \right] \leq 9 \Eb \left[ \frac{(\lambda_j^2 - V_j)^2}{\Ypp_j^4} \cdot \1_{ \mho_j } \right],
	\end{equation*}
	and since the definition of $\mho_j$ implies that $\Ypp_j^{-4} \leq \frac{81}{16}\lambda_j^{-4} \leq \frac{81}{16} d^4 \alpha_j^{-4}$, we have
	\begin{align*}
		\Eb \left[ \left( \frac{\lambda_j^2}{V_j} - 1 \right)^2 \1_{\Omega_j} \1_{\mho_j} \right] \leq \frac{729}{16} d^4 \alpha_j^{-4} \Eb [ (\lambda_j^2 - V_j)^2 ] &= \frac{729}{16} d^4 \alpha_j^{-4} (8 \sigma^4 + 8\sigma^2\lambda_j^2)\\
		&\leq \frac{729}{2} d^4 \sigma^4 \alpha_j^{-4} + \frac{729}{2} d^6 \sigma^2 \alpha_j^{-2}.
	\end{align*}
	We now turn to the summand with $\1_{\mho_j^\complement}$.
	First by the Cauchy-Schwarz inequality,
	\begin{align*}
		\Eb \left[ \left( \frac{\lambda_j^2}{V_j} - 1 \right)^2 \1_{\Omega_j} \1_{\mho_j^\complement} \right] &\leq \left( \Eb \left[ \left( \frac{\lambda_j^2}{V_j} - 1 \right)^4 \1_{\Omega_j} \right] \right)^{1/2} \cdot \Pb(\mho_j^\complement)^{1/2}\\
		&\leq \sigma^{-4} \cdot [\Eb (\lambda_j^2 - V_j)^4]^{1/2} \cdot \Pb(\mho_j^\complement)^{1/2}.
	\end{align*}
	Now, simple but exhausting calculations show that $\Eb [(\lambda_j^2 - V_j)^4] = 196\lambda_j^4\sigma^4 + 1920\lambda_j^2 \sigma^6 + 960 \sigma^8$.
	Thus, using the estimate $\sqrt{a+b+c} \leq \sqrt a + \sqrt b + \sqrt c$ for $a,b,c \geq 0$, we obtain
	\begin{align*}
		\Eb \left[ \left( \frac{\lambda_j^2}{V_j} - 1 \right)^2 \1_{\Omega_j} \1_{\mho_j^\complement} \right] &\leq \sigma^{-4} ( \sqrt{196} \lambda_j^2 \sigma^2 + \sqrt{1920} \lambda_j \sigma^3 + \sqrt{960} \sigma^4 ) \Pb(\mho_j^\complement)^{1/2}\\
		&\leq ( \sqrt{196}d^2 \alpha_j^2 \sigma^{-2} + \sqrt{1920} d\alpha_j \sigma^{-1} + \sqrt{960} ) \Pb(\mho_j^\complement)^{1/2}.
	\end{align*}
	By definition, $\mho_j^\complement = \{ \vert \lambda_j/\Ypp_j - 1 \vert > 1/2 \}$, which implies that $\lambda_j/\Ypp_j \notin [1/2,3/2]$ on $\mho_j^\complement$.
	Hence $\Ypp_j/\lambda_j \notin [2/3,2]$ on $\mho_j^\complement$ showing the inclusion $\mho_j^\complement \subseteq \{ \vert \Ypp_j/\lambda_j - 1 \vert > 1/3 \} = \{ \vert \Ypp_j - \lambda_j \vert > \vert \lambda_j \vert / 3 \}$, and hence
	\begin{equation*}
	 \Pb(\mho_j^\complement) \leq 2 \exp(-\lambda_j^2/(36\sigma^2)) \leq 2 \exp(-\alpha_j^2/(36d^2\sigma^2)).
	\end{equation*}
	We obtain
	\begin{align*}
		\Eb \left[ \left( \frac{\lambda_j^2}{V_j} - 1 \right)^2 \1_{\Omega_j} \1_{\mho_j^\complement} \right]&\\
		&\hspace{-6em}\leq (\sqrt{392}d^2 \alpha_j^2 \sigma^{-2} + \sqrt{3840} d\alpha_j \sigma^{-1} + \sqrt{1920}) \exp(-\alpha_j^2/(72d^2\sigma^2)).
	\end{align*}
	It is easy to see that there are constants $C_1(d), C_2(d)$ and $C_3(d)$ such that
	\begin{align*}
		&\alpha_j^2 \sigma^{-2} \exp(-\alpha_j^2/(72d^2\sigma^2)) \leq C_1(d) \sigma^2 \alpha_j^{-2},\\
		&\alpha_j \sigma^{-1} \exp (- \alpha_j^2/(72d^2\sigma^2)) \leq C_2(d)\sigma^2 \alpha_j^{-2},\\
		&\exp(-\alpha_j^2/(72d^2\sigma^2)) \leq C_3(d) \sigma^2 \alpha_j^{-2},
	\end{align*}
	and thus
	\begin{equation*}
		\Eb \left[ \left( \frac{\lambda_j^2}{V_j} - 1 \right)^2 \1_{\Omega_j} \1_{\mho_j^\complement} \right] \leq C(d) \sigma^2 \alpha_j^{-2}.
	\end{equation*}
	Now, combining the derived bounds for the two terms on the right hand-side of \eqref{eqPropDiagUpperAux:2} implies the claim assertion.
	For the proof of \ref{UpperAux3}, we consider first the case that $\lambda_j^2 \geq 12\sigma^2$.
	Then, by Chebyshev's inequality,
	\begin{equation*}
		 \Pb (\Omega_j^\complement) \leq \Pb \left( \frac{\Ypp_j^2}{\lambda_j^2} < \frac{1}{4} \right) \leq \Pb \left( \left| \frac{\Ypp_j}{\lambda_j} - 1 \right|  > \frac{1}{2} \right) \leq 8\sigma^2\lambda_j^{-2} \leq 8d^2\sigma^2 \alpha_j^{-2}.
	\end{equation*}
	In case that $\lambda_j^2 \leq 12\sigma^2$, we have $1 \leq 12d^2\sigma^2 \alpha_j^{-2}$ and $\Pb(\Omega_j^\complement) \leq 12d^2\sigma^2 \alpha_j^{-2}$ holds trivially.
	Combining the two considered cases implies the claim assertion.
\end{proof} \section{Calculations of rates}\label{appCalcRates}

We sketch the calculations leading to the rates in Table~\ref{tableRatesQuadFun}.
Recall that $\omega_j \equiv 1$ for all these examples.

\begin{itemize}
 \item 1. Case: $\gamma_j = j^p$, $\alpha_j=j^{-a}$
 
 It holds $\epsilon^4 \sum_{j=1}^k \alpha_j^{-4} = \epsilon^4 \sum_{j=1}^k j^{4a} \asymp \epsilon^4 k^{4a+1}$, and
 \begin{equation*}
  \epsilon^2 \max_{j \in \llbracket 1, k \rrbracket} \alpha_j^{-2} \gamma_j^{-2} = \epsilon^2 \max_{j \in \llbracket 1, k \rrbracket} j^{2a-2p} = \begin{cases} \epsilon^2, & \text{ if }p \geq a,\\ \epsilon^2 k^{2a-2p}, & \text{ if } p < a. \end{cases}
 \end{equation*}
  Thus, $\kepsilon \asymp \epsilon^{-4/(4a+4p+1)} \wedge \epsilon^{-1/(2p)}$.
  Similarly,
  \begin{equation*}
  \sigma^2 \max_{j \in \llbracket 1, k \rrbracket} \alpha_j^{-2} \gamma_j^{-4} = \sigma^2 \max_{j \in \llbracket 1, k \rrbracket} j^{2a-4p} = \begin{cases} \sigma^2, & \text{ if }2p \geq a,\\ \sigma^2 k^{2a-4p}, & \text{ if } 2p < a. \end{cases}
  \end{equation*}
  Hence, $\ksigma \asymp \sigma^{-1/a} \wedge \sigma^{-1/(2p)}$.
  The rate resulting from these values of $\kepsilon$ and $\ksigma$ is
  \begin{equation*}
	\epsilon^{16p/(4a+4p+1)} \wedge \epsilon^2 \wedge \sigma^{4p/a} \wedge \sigma^2.
  \end{equation*}
  
  \item 2. Case: $\gamma_j \asymp \exp(pj)$, $\alpha_j=j^{-a}$
  
  As in the previous case we have $\epsilon^4 \sum_{j=1}^k \alpha_j^{-4} \asymp \epsilon^4 k^{4a+1}$, but now
  \begin{equation*}
   \epsilon^2 \max_{j \in \llbracket 1, k \rrbracket} \alpha_j^{-2} \gamma_j^{-2} \lesssim \epsilon^2.
  \end{equation*}
  Balancing the approximation error $\gamma_k^{-4}$ and $\epsilon^2$ results in choosing $\kepsilon = \lfloor \lvert \log \epsilon \rvert /(2p) \rfloor$ which in turn leads to the parametric rate $\epsilon^2$.
  Analogously, $\ksigma = \lfloor \lvert \log \sigma \rvert /(2p) \rfloor$, likewise implying the parametric rate $\sigma^2$.
  \item 3. Case: $\gamma_j = j^p$, $\alpha_j \asymp \exp(-aj)$
  
  In this case, $\epsilon^4 \sum_{j=1}^k \alpha_j^{-4} \asymp \epsilon^4 \sum_{j=1}^k \exp(4ja) \asymp \epsilon^4 \exp(4ka)$, and balancing this expression with the approximation error $\gamma_k^{-4}$ leads to a choice of $\kepsilon$ of order $\lvert \log \epsilon \rvert$.
  Plugging this choice into the approximation error leads, with respect to $\epsilon$, to the rate $\lvert \log \epsilon \rvert^{-4p}$.
  Moreover, in the case at hand,
  \begin{equation*}
   \sigma^2 \max_{j \in \llbracket 1,k\rrbracket} \alpha_j^{-2} \gamma_j^{-4} \lesssim \sigma^2 \exp(2ka) \cdot k^{-4p}
  \end{equation*}
  which implies analogously to the choice of $\kepsilon$ a choice of $\ksigma$ of order $\lvert \log \sigma \rvert$.
  Hence, the rate with respect to $\sigma$ is $\lvert \log \sigma \rvert^{-4p}$.
  Note that the exact knowledge of $p$ and $a$ is not necessary in this case, since it suffices to choose $\kepsilon \asymp \lvert \log \epsilon \rvert$ and $\ksigma \asymp \lvert \log \sigma \rvert$ (however, one has to know that the unknown solution belongs to a Sobolev class and that the inverse problem is severely ill-posed).

  \item 4. Case: $\gamma_j \asymp \exp(pj)$, $\alpha_j \asymp \exp(-aj)$
  
  We have
  \begin{align*}
   \epsilon^2 \max_{j \in \llbracket 1,k\rrbracket} \alpha_j^{-2} \gamma_j^{-2} &\asymp \epsilon^2 \max_{j \in \llbracket 1,k\rrbracket} \exp(2(a-p)j)\\
   &= \begin{cases}
       \epsilon^2, & \text{ if } p \geq a,\\
       \epsilon^2 \exp(2(a-p)k), & \text{ if } p < a.
   \end{cases}
  \end{align*}
  Hence, $\kepsilon = \lfloor \lvert \log \epsilon \rvert /(a+p) \rfloor \wedge \lfloor \lvert \log \epsilon \rvert /(2p) \rfloor$ (the second choice would equally originate from balancing the term $\epsilon^4 \sum_{j=1}^k \alpha_j^{-4}$ with the approximation term $\gamma_k^{-4}$).
  Thus, the resulting rate in terms of $\epsilon$ is $\epsilon^{4p/(p+a)} \vee \epsilon^2$.
  Concerning the rate in terms of $\sigma^2$ for this case, note first that
  \begin{align*}
   \sigma^2 \max_{j \in \llbracket 1,k\rrbracket} \alpha_j^{-2} \gamma_j^{-4} &\asymp \sigma^2 \max_{j \in \llbracket 1,k\rrbracket} \exp(2(a-2p)j)\\ &= \begin{cases}
       \sigma^2, & \text{ if } 2p \geq a,\\
       \sigma^2 \exp(2(a-p)k), & \text{ if } 2p < a.
   \end{cases}
  \end{align*}
  Thus, $\ksigma = \lfloor \lvert \log \sigma \rvert/a \rfloor \wedge \lfloor \lvert \log \sigma \rvert/(2p) \rfloor$ leading to the rate $\sigma^2 \vee \sigma^{4p/a}$ in terms of $\sigma$.
\end{itemize}
 \section{Proofs of Section~\ref{secLower}}\label{appLower}

\subsection{Proof of Theorem~\ref{thmLowerEps1}}
By Markov's inequality one has for every estimator $\qtilde$ of $\qfunc(\theta)$ that
\begin{equation}\label{eqRedEps1}
        \inf_\qtilde \sup_{\theta \in \csol} \sup_{\lambda \in \cev} \Eb [ ( \qtilde - \qfunc(\sol) )^2 ] \geq \psi^2 \cdot \inf_\qtilde \sup_{\theta \in \csol} \sup_{\lambda \in \cev} \Pb (( \qtilde - \qfunc(\theta) )^2 \geq \psi^2),
\end{equation}
and we want to apply Lemma~\ref{lemLowerHypercube} from Appendix~\ref{appGenToolsLowerBounds} with $\psi = \frac{1}{2}L^2 \nu^{-1/2} \epsilon^2 \sqrt{\sum_{j=1}^{\kappa} \omega_j^4\alpha_j^{-4}}$.
For any $\tau = (\tau_1,\ldots,\tau_\kappa) \in \setpm^\kappa$ define $\sol^\tau$ via
\begin{equation*}
	\sol^\tau_i = \tau_i \cdot L\nu^{-1/4} \cdot \epsilon \cdot \frac{\omega_i \alpha_i^{-2}}{(\sum_{j=1}^{\kappa} \omega_j^4\alpha_j^{-4})^{1/4}} \qquad \text{for } i \in \llbracket 1,\kappa\rrbracket,
\end{equation*}
and $\sol^\tau_i = 0$ for $i > \kappa$.
Then, for any $\tau \in \setpm^\kappa$,
\begin{equation*}
	\sum_{j=1}^{\infty} (\theta_j^\tau)^2 \gamma_j^2 = L^2 \nu^{-1/2} \cdot\frac{\epsilon^2}{(\sum_{i=1}^{\kappa} \omega_i^4\alpha_i^{-4})^{1/2}} \ \sum_{j=1}^{\kappa} \omega_j^2\alpha_j^{-4} \gamma_j^2 \leq L^2
\end{equation*}
where we have used that $\epsilon^2 \omega_j^{-2}\gamma_j^2 \leq \epsilon^2 \omega_\kappa^{-2} \gamma_\kappa^2 \leq \nu^{1/2}(\sum_{i=1}^\kappa \omega_i^4\alpha_i^{-4})^{-1/2}$ by assumption.
Thus, $\theta^\tau \in \csol$ for any $\tau \in \setpm^\kappa$.
Further,
\begin{equation*}
	\qfunc(\sol^\tau) = L^2 \nu^{-1/2} \epsilon^2 \frac{\sum_{j=1}^{\kappa} \omega_j^4/\alpha_j^{4}}{\sqrt{\sum_{j=1}^{\kappa} \omega_j^4/\alpha_j^{4}}} = L^2 \nu^{-1/2} \epsilon^2 \sqrt{\sum_{j=1}^{\kappa} \omega_j^4/\alpha_j^{4}}.
\end{equation*}
Consider the probability measure $\mu$ on $\csol$ that is induced by the uniform distribution on the hypercube $\setpm^\kappa$ via the mapping $\setpm^\kappa \to \csol, \, \omega \mapsto \sol^\tau$.
Let $\Pb_\mu$ be the resulting distribution of the tuple $(X,Y)$ when $\lambda=\lambda^\circ$ for some fixed but arbitrary $\lambda^\circ \in \cev$, and analogously $\Pb_0$ the distribution of $(X,Y)$ when $\sol = 0 \in \csol$ and $\lambda = \lambda^\circ$.
Computing the $\chi^2$-distance between $\Pb_\mu$ and $\Pb_0$ yields
\begin{align*}
	\chi^2(\Pb_\mu, \Pb_0) = \int \left( \frac{\dd \Pb_\mu}{\dd \Pb_0} \right)^2 \dd \Pb_0 - 1 &= \prod_{j=1}^{\kappa} \frac{\exp(-\lambda_j^2\beta_j^2/\epsilon^2) + \exp(\lambda_j^2\beta_j^2/\epsilon^2)}{2}  - 1
\end{align*}
where $\beta_j = \epsilon L\nu^{-1/4} \cdot \frac{\omega_j\alpha_j^{-2}}{(\sum_{i=1}^{\kappa} \omega_i^4\alpha_i^{-4})^{1/4}}$.
Now, using the same reasoning as on page~130 in~\cite{tsybakov2009introduction}, it can be shown that there exists a constant $c_2 < \infty$ such that
\begin{equation*}
	\frac{\exp(-\lambda_j^2\beta_j^2/\epsilon^2) + \exp(\lambda_j^2\beta_j^2/\epsilon^2)}{2} \leq \exp \left(\frac{c_2\lambda_j^4 \beta_j^4}{\epsilon^4} \right).
\end{equation*}
Thus, denoting with $c_2$ and $c_3$ numerical constants that depend on $d$, we conclude
\begin{equation*}
	\chi^2(\Pb_\mu, \Pb_0) \leq \exp \left(c_2 \sum_{j=1}^{\kappa} \frac{\alpha_j^{4} \beta_j^4}{\epsilon^4} \right) - 1 \leq \exp(c_3) - 1
\end{equation*}
by definition of $\beta_j$.
Hence, all the assumptions of Lemma~\ref{lemLowerHypercube} are satisfied.
Application of this lemma together with~\eqref{eqRedEps1} implies
\begin{equation*}
	\inf_\qtilde \sup_{\sol \in \csol} \sup_{\ev \in \cev} \Pb( \vert \qtilde - \qfunc(\sol) \vert \geq \psi) \geq \frac{1}{4} \exp(-\beta)
\end{equation*}
where $\beta = \exp(c_3) - 1$, and putting this into~\eqref{eqRedEps1} implies the claim assertion.

\subsection{Proof of Theorem~\ref{thm:lower:epsilon:2}}\label{appLowerEps2}
For the proof of Statement \ref{lower:epsilon:2:a} we define for $\tau \in \setpm$ hypotheses $(\sol^\tau,\ev^\tau) \in \csol \times \cev$ with $\ev^1 = \ev^{-1} = \ev^\circ$ for some arbitrary but fixed $\ev^\circ \in \cev$.
Putting $\zeta = \min \{ 1/2, \sqrt{2}/(Ld\sqrt{\nu}) \} $, for $\tau \in \setpm$ the hypotheses concerning the solution are defined as $\sol^\tau = (\sol^{\tau}_j)_{j \in \N}$ where
\begin{equation*}
	\sol^{\tau}_\kappa = \frac{L}{2} \big( 1 + \tau \zeta \big)  \gamma_\kappa^{-1}
\end{equation*}
and $\sol^{\tau}_j = 0$ for $\tau \in \setpm$ and $j \neq \kappa$.
Then, $\sol^\tau \in \csol$ for $\tau \in \setpm$ since
\begin{equation*}
	\sum_{j=1}^\infty (\sol_j^{\tau})^2 \gamma_j^2 = (\sol_\kappa^{\tau})^2 \gamma_\kappa^2 \leq \frac{L^2}{4} \cdot 4 \gamma_\kappa^{-2} \gamma_\kappa^2 = L^2.
\end{equation*} 	
Denote by $\Pb_\tau$ the distribution of the tuple $(X,Y)$ if the true parameter is $(\sol^\tau, \lambda^\tau) = (\sol^\tau, \lambda^\circ)$.
Then, the Kullback-Leibler distance between $\Pb_1$ and $\Pb_{-1}$ depends only on the marginal distributions $\Pb_1^X$ and $\Pb_{-1}^X$, and we have by definition of $\nu$ and $\zeta$ that
\begin{equation*}
	\KL(\Pb_1, \Pb_{-1}) = \frac{1}{2\epsilon^2} \cdot (\zeta L  \ev_\kappa^\circ \gamma_\kappa^{-1})^2 \leq \frac{(\zeta L d \alpha_\kappa \gamma_\kappa^{-1})^2}{2\epsilon^2} \leq 1.
\end{equation*}
Now
\begin{equation*}
	\qfunc_1 - \qfunc_{-1} = \frac{L^2}{4} (1 + \zeta )^2 \omega_\kappa^2\gamma_\kappa^{-2} - \frac{L^2}{4} (1 - \zeta )^2 \omega_\kappa^2\gamma_\kappa^{-2} = L^2 \zeta \omega_\kappa^2 \gamma_\kappa^{-2},
\end{equation*}
and \eqref{eqRedSch} implies
\begin{equation*}
	\sup_{\sol \in \csol} \sup_{\ev \in \cev} \Eb [(\qtilde - \qfunc(\sol))^2] \geq \frac{1}{16} L^4 \zeta^2 \omega_\kappa^4 \gamma_\kappa^{-4}
\end{equation*}
which implies Statement~\ref{lower:epsilon:2:a} (again by definition of $\nu$).
	
For the proof of the parametric rate $\epsilon^2$ in \ref{thm:lower:epsilon:2:b} we use the same approach as in \ref{lower:epsilon:2:a} but define the two hypotheses $\sol^\tau = (\sol^\tau_{j})_{j \in \N}$, $\tau \in \setpm$ by $\sol_{1}^\tau = \left( 1 + \tau \epsilon \right) \cdot \zeta$ with $\zeta = \min \{ L/2, 1/(\sqrt{2}d) \}$, and $\sol_{j}^\tau = 0$ for $j \geq 2$.
	Then, $\sol^\tau \in \csol$ since
	\begin{equation*}
		\sum_{j=1}^\infty (\sol^{\tau}_j)^2 \gamma_j^2 = \left( 1 + \tau \epsilon \right)^2 \zeta^2 \leq 4 \cdot \frac{L^2}{4} = L^2
	\end{equation*}
	(recall that we assume $\epsilon \leq 1$ throughout the paper), and the Kullback-Leibler distance between $\Pb_1$ and $\Pb_{-1}$ satisfies
        \begin{equation*}
		\KL(\Pb_1, \Pb_{-1}) = \frac{( \ev_1^\circ (\sol_{1}^1 - \sol^{-1}_1) )^2}{2\epsilon^2} \leq 2d^2\zeta^2 \leq 1.
	\end{equation*}
	Since $\qfunc_1 - \qfunc_{-1} = 4\zeta^2\epsilon$, the reduction scheme~\eqref{eqRedSch} implies
	\begin{equation*}
		\sup_{\sol \in \csol} \sup_{\ev \in \cev} \Eb [ (\qtilde - \qfunc(\sol))^2] \geq \zeta^4 \epsilon^2,
	\end{equation*}
	and the statement follows since $\qtilde$ is arbitrary.

\subsection{Proof of Theorem~\ref{thmLowerSigma}}
As in the proof of Theorem~\ref{thm:lower:epsilon:2}, for the proof of both parts \ref{thmLowerSigma:a} and \ref{thmLowerSigma:b} we will use the reduction scheme described in Section~\ref{ssecReduction2} in the appendix wherefrom we borrow also the notation used in the rest of the proof.
	In order to prove \ref{thmLowerSigma:a} define for $\tau \in \setpm$ hypotheses $(\sol^\tau, \ev^\tau) \in \csol \times \cev$ by means of
	\begin{align*}
		&\sol_{\kappa}^\tau = Ld^{-1} (1 + \tau \zeta) \gamma_\kappa^{-1},  \quad \text{and} \quad \sol_{j}^\tau = 0 \text{ for } j \neq \kappa,\\
		&\ev_{\kappa}^\tau = (1- \tau \zeta) \alpha_\kappa, \quad \text{and} \quad \ev_{j}^\tau = \alpha_j \text{ for } j \neq \kappa,
	\end{align*}
	where we put $\zeta = \min \{ 1/\sqrt{2\nu}, 1 - d^{-1} \}$.
	Note that the estimate $d^{-2} \leq (1 - \zeta)^2 \leq 1 \leq (1 + \zeta)^2 \leq d^2$ holds where the last inequality follows by the inequality $2d-1\leq d^2$ which is true for $d \geq 1$.
	First, $\sol^\tau \in \csol$ for $\tau \in \setpm$ because
	\begin{equation*}
		\sum_{j=1}^{\infty} (\sol_{j}^\tau)^2 \gamma_j^2 = L^2 d^{-2} (1 + \tau \zeta)^2 \leq L^2.
	\end{equation*}
	Moreover $\ev \in \cev$, since
	\begin{equation*}
		\frac{1}{d^2} \alpha_\kappa^2 \leq (1 - \zeta)^2 \alpha_\kappa^2 \leq \alpha_\kappa^2 \leq (1 + \zeta)^2 \alpha_\kappa^2 \leq d^2 \alpha_\kappa^2,
	\end{equation*}
	and $d^{-2} \alpha_j^2 \leq (\lambda_j^\circ)^2 \leq d^2 \alpha_j^2$ for all $j\neq \kappa$ holds trivially.
	Note that $\sol^1 \ev^1 = \sol^{-1} \ev^{-1}$ by construction, and hence the Kullback-Leibler distance between $\Pb_1$ and $\Pb_{-1}$ depends only on the distance between the marginals $\Pb_1^{Y_\kappa}$ and $\Pb_{-1}^{Y_\kappa}$.
	Thus, by definition of $\zeta$
	\begin{equation*}
		\KL(\Pb_1,\Pb_{-1}) = \KL(\Pb_1^{Y_\kappa}, \Pb_{-1}^{Y_\kappa}) = \frac{1}{2\sigma^2} \cdot \left( 2 \zeta \alpha_\kappa \right)^2 \leq \frac{2\zeta^2 \alpha_\kappa^2}{\sigma^2} \leq 1.  
	\end{equation*}
	Further, it holds $\qfunc_1 - \qfunc_{-1} = \frac{L^2}{d^2} (1+\zeta)^2 \omega_\kappa^2 \gamma_\kappa^{-2} - \frac{L^2}{d^2} (1-\zeta)^2 \omega_\kappa^2 \gamma_\kappa^{-2}= \frac{4L^2}{d^2} \zeta \omega_\kappa^2\gamma_\kappa^{-2}$, and by applying \eqref{eqRedSch} we obtain
	\begin{equation*}
	\sup_{\sol \in \csol} \sup_{\ev \in \cev} \Eb [ ( \qtilde - \qfunc(\sol) )^2 ] \geq \frac{L^4}{d^4} \zeta^2 \omega_\kappa^4 \gamma_\kappa^{-4}.
	\end{equation*}
	Now \ref{thmLowerSigma:a} follows since $\sigma^2 \alpha_\kappa^{-2} \asymp_\nu 1$ and $\qtilde$ was arbitrary.
	For the proof of statement \ref{thmLowerSigma:b}, introduce for $\tau \in \setpm$ the hypotheses $(\sol^\tau, \ev^\tau) \in \csol \times \cev$ defined by
	\begin{align*}
		&\sol_{1}^\tau = (1+\tau \sigma \zeta) \frac{L}{2}, \qquad \text{and} \qquad \sol_{j}^\tau = 0 \text{ for } j \geq 2,\\
		&\ev_{1}^\tau = (1-\tau \sigma \zeta) , \qquad \text{and} \qquad \ev_{j}^\tau = \alpha_j \text{ for } j \geq 2 
	\end{align*}
	where $\zeta = \min  \{1/\sqrt{2}, 1-d^{-1} \}$.
	Then, grant to $\sigma \leq 1$, $\sol^\tau \in \csol$ follows from the calculation
	\begin{equation*}
		\sum_{j=1}^{\infty} (\sol_{j}^\tau)^2 \gamma_j^2 \leq (1+\tau \sigma \zeta)^2 \cdot \frac{L^2}{4} \leq L^2, 
	\end{equation*}
	and the inequality $\frac{1}{d^2} \leq (1-\sigma \zeta)^2 \leq 1 \leq (1 + \sigma \zeta)^2 \leq d^2$ shows that $\ev \in \cev$.
	By construction the Kullback-Leibler distance between $\Pb_1$ and $\Pb_{-1}$ depends only on the marginal distributions of $Y_1$, and hence
	\begin{equation*}
		\KL(\Pb_1, \Pb_{-1}) = \frac{1}{2\sigma^2} \big( \ev_{1}^1 - \ev_{1}^{-1} \big)^2 = \frac{1}{2\sigma^2} \cdot 4\sigma^2 \zeta^2 \leq 2\zeta^2 \leq 1.
	\end{equation*}
	Noting that $\qfunc_1 - \qfunc_{-1} = \sigma \zeta L^2$ we conclude from \eqref{eqRedSch} that
	\begin{equation*}
		\sup_{\sol \in \csol} \sup_{\ev \in \cev} \Eb [ ( \qtilde - \qfunc(\sol) )^2 ] \geq \frac{L^4}{16} \zeta^2 \sigma^2
	\end{equation*}
	which implies the claim assertion since $\qtilde$ was arbitrary. \section{Proofs of Section~\ref{secTesting}}

\subsection{Proof of Theorem~\ref{UpperSignalDetection}}
Consider the test statistic defined in~\eqref{eqDefQhatsd} with
	\begin{equation*}
	 \Ctilde \defeq \max \{ \sqrt 8 \delta^{-1/2}, 32d^2 \delta^{-1} \}.
	\end{equation*}
	Let us first show that the type I error can be bounded by $\delta/2$.
	Indeed, by a direct calculation, one has
	\begin{equation*}
	    \Pb_0 (\Deltahatsd = 1) = \Pb_0(\qhatksd \geq \Ctilde \phi_\epsilon^2) \leq \frac{\Eb_0 [\qhatksd^2]}{\Ctilde^2 \phi_\epsilon^4} \leq \frac{2 \epsilon^4 \sum_{j=1}^{\ksd} \alpha_j^{-4}}{\Ctilde^2 \phi_\epsilon^4} \leq \delta/2.
	\end{equation*}
	where we used $\Ctilde \geq 2 \delta^{-1/2}$.
	
	In order to bound the type II error, let $\sol \in \setAlt(C\phi_\epsilon^2)$ be arbitrary.
	We distinguish two cases.
	
	\emph{Case 1: $\sum_{j=1}^{\ksd} \sol_j^2 \geq 2 d^2\Ctilde \phi_\epsilon^2$}. In this case we have
	\begin{align*}
	 \Pb_\sol(\Deltahatsd = 0) = \Pb_\sol(\qhatksd \leq \Ctilde \phi_\epsilon^2) &\leq \Pb_\sol\Big(\qhatksd \leq \frac{1}{2d^2} \sum_{j=1}^{\ksd} \sol_j^2\Big)\\
	 &\leq \Pb_\sol\Big(\qhatksd - \Eb_\sol \qhatksd \leq -\frac{1}{2d^2} \sum_{j=1}^{\ksd} \sol_j^2\Big)
	\end{align*}
	where we have used that $\Eb_\sol \qhatksd \geq \sum_{j=1}^{\ksd} \sol_j^2 /d^2$.
	Thus,
	\begin{align*}
		\Pb_\sol(\Deltahatsd = 0) &\leq  \frac{4d^4\Eb_\sol [(\qhatksd - \Eb \qhatksd)^2]}{( \sum_{j=1}^{\ksd} \sol_j^2 )^2}\\
		&\leq 4d^4 \left\lbrace \frac{2 \epsilon^4 \sum_{j=1}^{\ksd} \alpha_j^{-4}}{( \sum_{j=1}^{\ksd} \sol_j^2 )^2} + \frac{4\epsilon^2 \sum_{j=1}^{\ksd} \alpha_j^{-4} (\lambda_j\sol_j)^2}{( \sum_{j=1}^{\ksd} \sol_j^2 )^2} \right\rbrace\\
		&\leq 4d^4 \left\lbrace \frac{2}{4d^4 \Ctilde^2} + \frac{4\epsilon^2\alpha_{\ksd}^{-2} \sum_{j=1}^{\ksd}\sol_j^2 }{( \sum_{j=1}^{\ksd} \sol_j^2 )^2} \right\rbrace\\
		&\leq \frac{2}{\Ctilde^2} + \frac{16d^4\epsilon^2\alpha_{\ksd}^{-2}}{2d^2 \Ctilde \phi_\epsilon^2}\\
		&\leq \frac{2}{\Ctilde^2} + \frac{8d^2}{\Ctilde}\\
		&\leq \delta/2
	\end{align*}
	where the last estimate is due to $\Ctilde \geq \max \{ \sqrt 8 \delta^{-1/2}, 32d^2 \delta^{-1} \}$.
	Since $\sol \in \setAlt$ was arbitrary, this shows that the type II error can be bounded from above by $\delta/2$ in this case.
	
	\emph{Case 2: $\sum_{j=1}^{\ksd} \sol_j^2 \leq 2 d^2\Ctilde \phi_\epsilon^2$}.
	First note that, by definition of $\setAlt$,
	\begin{align*}
		\Eb_\sol \qhatksd \geq \frac{1}{d^2} \sum_{j=1}^{\ksd} \sol_j^2 &= \frac{1}{d^2} \left\{ \sum_{j=1}^{\infty} \sol_j^2 - \sum_{j > \ksd} \sol_j^2 \right\}\\
		&\geq \frac{1}{d^2} \bigg( C^2 \phi_\epsilon^2 - \sum_{j>\ksd} \frac{\gamma_j^2}{\gamma_j^2} \sol_j^2 \bigg) \\
		&\geq \bigg( \frac{C^2}{d^2} - \frac{L^2\sqrt{\nu}}{d^2} \bigg) \phi_\epsilon^2.
	\end{align*}
	Now, since $\epsilon^2 \sum_{j=1}^{\ksd} \alpha_j^{-2} \sol_j^2 \leq \epsilon^2 \alpha_{\ksd}^{-2} \sum_{j=1}^{\ksd} \sol_j^2 \leq 2d^2 \Ctilde \phi_\epsilon^4$ in Case~2, we obtain, choosing $C$ such that $C^2/d^2 - L^2\sqrt{\nu}/d^2 - \Ctilde > 0$,
	\begin{align*}
		\Pb_\sol(\Deltahatsd = 0) &= \Pb_\sol(\qhatksd - \Eb_\sol \qhatksd \leq \Ctilde \phi_\epsilon^2 - \Eb_\sol \qhatksd) \\
		&\leq \Pb_\sol(\Eb_\sol \qhatksd - \qhatksd \geq (C^2/d^2 - L^2\sqrt{\nu}/d^2 - \Ctilde) \phi_\epsilon^2)\\
		&\leq \frac{\var (\qhatksd)}{\left( \frac{C^2}{d^2} - \Ctilde - \frac{L^2\sqrt{\nu}}{d^2} \right)^2 \phi_\epsilon^4}\\
		&\leq \frac{2 \epsilon^4 \sum_{j=1}^{\ksd} \alpha_j^{-4}}{\left( \frac{C^2}{d^2} - \Ctilde - \frac{L^2\sqrt{\nu}}{d^2} \right)^2 \phi_\epsilon^4} + \frac{4d^2\epsilon^2 \sum_{j=1}^{\ksd} \alpha_j^{-2}\sol_j^2 }{\left( \frac{C^2}{d^2} - \Ctilde - \frac{L^2\sqrt{\nu}}{d^2} \right)^2 \phi_\epsilon^4}\\
		&\leq \frac{2}{\left( \frac{C^2}{d^2} - \Ctilde - \frac{L^2\sqrt{\nu}}{d^2} \right)^2} + \frac{8d^4\Ctilde}{\left( \frac{C^2}{d^2} - \Ctilde - \frac{L^2\sqrt{\nu}}{d^2} \right)^2}
	\end{align*}
	and the last expression is bounded from above by $\delta/2$ for $C$ sufficiently large.
	Thus, the type II error is bounded by $\delta/2$ also in Case~2 and the statement of the proposition follows.
	
\subsection{Proof of Theorem~\ref{thmLowerSignalDetection}}

In order to prove the theorem, we will use Statement~\ref{lemLowerTesting1} from Lemma~\ref{lemLowerTesting}.
	For any $\tau \in \setpm^{\ksd}$ define $\sol^\tau$ by
	\begin{equation*}
		\sol_i^\tau = \tau_i \epsilon c \cdot \frac{\alpha_i^{-2}}{( \sum_{j=1}^{\ksd} \alpha_j^{-4} )^{1/4}}, \quad i \in \llbracket 1,\ksd\rrbracket,
	\end{equation*}
	and $\sol_i^\tau = 0$ for $i > \ksd$.
	Then, in analogy to the proof of Theorem~\ref{thmLowerEps1}, it can be shown that $\sol^\tau \in \csol$ provided that $c^2 \leq L^2\nu^{-1/2}$.
	Moreover, for all $\tau \in \setpm^{\ksd}$,
	\begin{align*}
		\qfunc(\theta^\tau) = c^2 \epsilon^2 \sqrt{\sum_{j=1}^{\ksd} \alpha_j^{-4 }} = c^2 \phi_\epsilon^2,
	\end{align*}
	and hence the law of $\ksd$ independent Rademacher random variables induces a probability distribution $\mu$ on the set $\setAlt(c\phi_\epsilon)$.
	Finally, again in analogy to the proof of Theorem~\ref{thmLowerEps1}, it holds
	\begin{equation*}
		\chi^2(\Pb_0, \Pb_\mu) \leq \exp(c_2c^2) - 1
	\end{equation*}
	for some fixed numerical constant $c_2=c_2(d) > 0$.
	Now, taking $c$ sufficiently small implies $\chi^2(\Pb_0, \Pb_\mu) \leq \delta^2$, and applying Lemma~\ref{lemLowerTesting} yields the claim assertion.

\subsection{Proof of Theorem~\ref{thmUpperGoF}}
We consider the test statistic defined in~\eqref{defDeltahatgof} where the conditions on $\Ctilde$ will be stated in the sequel. 
	We start by bounding the type I error from above by $\delta/2$:
	\begin{align*}
		\Pb_{0}(\Deltahatgof = 1) &= \Pb_{0}( \qhatkgof \geq \Ctilde \phi_{\epsilon, \sigma}^2 ) \leq \frac{\Eb_{0}[ \qhatkgof^2 ]}{\Ctilde^2 \phi_{\epsilon,\sigma}^4}.
	\end{align*}
	Now,
	\begin{align*}
		\Eb_{0}[ \qhat_{\kgof}^2 ] &\leq 672d^4 \epsilon^4 \sum_{j=1}^{\kgof}\alpha_j^{-4} + 2688d^4L^4 \sigma^4 \max_{j \in \llbracket 1,\kgof \rrbracket} \alpha_j^{-4}\gamma_j^{-4}\\
		&\leq (672d^4+2688 d^4L^4) \phi_{\epsilon,\sigma}^4,
	\end{align*}
	and hence $\Pb_{0}(\Deltahatgof = 1) \leq \delta/2$ provided that $\Ctilde^2 \geq 2(672d^4+2688d^4L^4) \delta^{-1}$.
	
	Now, we consider the type II error. In order to bound it from above by $\delta/2$, let $\sol \in \setAlt(C\phi_{\epsilon, \sigma})$ be arbitrary. It holds
	\begin{equation*}
		\Pb_\sol(\Deltahatgof = 0) = \Pb_\sol(\qhatkgof \leq \Ctilde \phi_{\epsilon, \sigma}^2) = \Pb_\sol( \qhatkgof - \sum_{j=1}^{\kgof} (\sol_j - \solcirc_j)^2 \leq \Ctilde \phi_{\epsilon, \sigma}^2 - \sum_{j=1}^{\kgof} (\sol_j - \solcirc_j)^2),
	\end{equation*}
	and as in the proof of Theorem~\ref{UpperSignalDetection} we consider two cases.
	
	\emph{Case 1: $\sum_{j=1}^{\kgof} (\sol_j - \solcirc_j)^2 \geq 2 \Ctilde \phi_{\epsilon, \sigma}^2$}.
	Then $\Ctilde \phi_{\epsilon,\sigma}^2 \leq \sum_{j=1}^{\kgof} (\sol_j - \solcirc_j)^2/2$, and thus
	\begin{align*}
		\Pb_\sol (\Deltahatgof = 0) &\leq \Pb_\sol \bigg(\qhat_{\kgof} - \sum_{j=1}^{\kgof} (\sol_j - \solcirc_j)^2 \leq -\sum_{j=1}^{\kgof} (\sol_j - \solcirc_j)^2/2\bigg)\\
		&=\Pb_\sol \bigg(-\qhat_{\kgof} + \sum_{j=1}^{\kgof} (\sol_j - \solcirc_j)^2 \geq \sum_{j=1}^{\kgof} (\sol_j - \solcirc_j)^2/2 \bigg)\\
		&\leq \frac{4\Eb_\sol [ ( \qhatkgof - \sum_{j=1}^{\kgof} (\sol_j - \solcirc_j)^2 )^2 ]}{( \sum_{j=1}^{\kgof} (\sol_j - \solcirc_j)^2 )^2}.
	\end{align*}
	Now, similarly as in the proof of Theorem~\ref{thmUpper},
	\begin{align*}
		\Eb_\sol [ ( \qhat_{\kgof} - \sum_{j=1}^{\kgof} (\sol_j - \solcirc_j)^2 )^2 ] &\leq 3 \sum_{i=1}^{3} \Eb \Tc_{\kgof i}^2
	\end{align*}
	where $\Tc_{\kgof 1}$, $\Tc_{\kgof 2}$, and $\Tc_{\kgof 3}$ are defined as in the proof of Theorem~\ref{thmUpper}.
	Following line by line the derivation of the upper bounds for the three terms on the right-hand side of the last display from Appendix~\ref{app:aux:upper}, we obtain
	\begin{align*}
		\Eb_\sol [ ( \qhat_{\kgof} - \sum_{j=1}^{\kgof} (\sol_j - \solcirc_j)^2 )^2 ] &\leq C(d,L) \phi_{\epsilon,\sigma}^4 + C(d,L)\phi_{\epsilon, \sigma}^2 \sum_{j=1}^{\kappa_2} (\sol_j - \solcirc_j)^2.
	\end{align*}
	Hence
	\begin{align*}
	  \Pb_\sol (\Deltahatgof = 0) &\leq \frac{C(d,L)\phi_{\epsilon, \sigma}^4}{\Ctilde^2 \phi_{\epsilon,\sigma}^4} + \frac{C(d,L) \phi_{\epsilon, \sigma}^2}{\sum_{j=1}^{\kgof} (\sol_j - \solcirc_j)^2 }\\
	  &\leq \frac{C(d,L)}{\Ctilde^2} + \frac{C(d,L)}{\Ctilde},
	\end{align*}
	and the last expression is smaller than $\delta/2$ for $\Ctilde$ sufficiently large\footnote{In order to make a lower bound on $\Ctilde$ explicit, it would be necessary to make the constants in Statement~\ref{UpperAux2} of Proposition~\ref{UpperAux} explicit, and we do not address this issue here.}.
	
	\emph{Case 2: $\sum_{j=1}^{\kgof} (\sol_j - \solcirc_j)^2 \leq 2 \Ctilde \phi_{\epsilon, \sigma}^2$}.
	First note that $\sol \in \setAlt(C\phi_{\epsilon, \sigma})$ implies
	\begin{align*}
		\sum_{j=1}^{\kgof} (\sol_j - \solcirc_j)^2 &= \sum_{j=1}^{\infty} (\sol_j - \solcirc_j)^2 - \sum_{j>\kgof} (\sol_j - \solcirc_j)^2\geq (C^2  - 2L^2\sqrt \nu) \phi_{\epsilon,\sigma}^2.
	\end{align*}
	Thus, for $C$ sufficiently large\footnote{\label{foot}Again, we are not able to give explicit bounds on $C$ due to the fact that the constant in Statement~\ref{UpperAux2} of Proposition~\ref{UpperAux} is not made explicit.},
	\begin{align*}
	\Pb_\sol(\Deltahatgof = 0) &\leq \Pb_\sol \bigg( \qhatkgof - \sum_{j=1}^{\kgof} (\sol_j - \solcirc_j)^2 \leq \Ctilde \phi_{\epsilon, \sigma}^2 - \sum_{j=1}^{\kgof} (\sol_j - \solcirc_j)^2 \bigg)\\
	&\leq \Pb_\sol\bigg( \qhatkgof - \sum_{j=1}^{\kgof} (\sol_j - \solcirc_j)^2 \leq (\Ctilde - C^2 + 2L^2 \sqrt \nu) \phi_{\epsilon, \sigma}^2 \bigg)\\
	&= \Pb_\sol\bigg( -\qhat_{\kgof} + \sum_{j=1}^{\kgof} (\sol_j - \solcirc_j)^2 \geq (C^2 - \Ctilde - 2L^2 \sqrt \nu) \phi_{\epsilon, \sigma}^2 \bigg)\\
	&\leq \frac{\Eb_\sol [(\qhatkgof - \sum_{j=1}^{\kgof} (\sol_j - \solcirc_j)^2)^2]}{(C^2 - \Ctilde - 2L^2 \sqrt \nu)^2 \phi_{\epsilon, \sigma}^4}.
	\end{align*}
	Now, as in the first case,
	\begin{align*}
		\Eb_\sol [ ( \qhat_{\kgof} - \sum_{j=1}^{\kgof} (\sol_j - \solcirc_j)^2 )^2 ] &\leq C(d,L) \phi_{\epsilon,\sigma}^4 + C(d,L)\phi_{\epsilon, \sigma}^2 \sum_{j=1}^{\kappa_2} (\sol_j - \solcirc_j)^2 \leq C(d,L) \phi_{\epsilon,\sigma}^4.
	\end{align*}
	Hence,
	\begin{align*}
	  \Pb_\sol(\Deltahatgof = 0) &\leq \frac{C(d,L) \phi_{\epsilon,\sigma}^4}{(C^2 - \Ctilde - 2L^2 \sqrt \nu)^2 \phi_{\epsilon, \sigma}^4}
	\end{align*}
	and $\Pb_\sol(\Deltahatgof = 0) \leq \delta/2$ provided that $C$ is sufficiently large\footnote{See Footnote~\ref{foot}.}.

\subsection{Proof of Theorem~\ref{LowerGoF}}
The case that $\phi_{\epsilon, \sigma}^2 = \epsilon^2 \sqrt{\sum_{j=1}^{\kgof} \alpha_j^{-4}}$ is dealt with in analogy to the proof of Theorem~\ref{thmLowerSignalDetection}, and thus omitted (the additional assumption $\phi_{\epsilon,\sigma}^4 \asymp_\nu \gamma_{\kgof}^{-4}$ is only exploited in this case).
	Thus, we consider the case $\phi_{\epsilon, \sigma}^2 = \sigma^2 \max_{j \in \llbracket 1,\kgof\rrbracket} \alpha_j^{-2}\gamma_j^{-2}$, and put
	$\kappa = \argmax_{j \in \llbracket 1,\kgof\rrbracket } \alpha_j^{-2}\gamma_j^{-2}$.
	We apply Statement~\ref{lemLowerTesting2} of Lemma~\ref{lemLowerTesting} to the testing problem
	\begin{equation*}
		\Hc_0\colon \sol = \sol^\circ, \lambda = \lambda^\circ \quad \text{against} \quad \Hc_1\colon \sol = \sol^1, \lambda = \lambda^1
	\end{equation*}	
	where $\sol^1_{j}=\sol^\circ_{j}$ for $j \neq \kappa$, $\sol^1_{\kappa} =\frac{1-\ctilde \sigma \alpha_{\kappa}^{-1}\gamma_{\kappa}^{-1}}{1+\ctilde \sigma \alpha_{\kappa}^{-1}\gamma_{\kappa}^{-1} } \cdot \sol^\circ_{\kappa}$, $\lambda^\circ_{j}=\lambda^1_{j}=\alpha_j$ for $j \neq \kappa$, $\lambda^\circ_{\kappa} = (1-\ctilde \sigma \alpha_{\kappa}^{-1}\gamma_{\kappa}^{-1} ) \alpha_{\kappa}$, and $\lambda^1_{\kappa} = (1+\ctilde \sigma \alpha_{\kappa}^{-1}\gamma_{\kappa}^{-1} ) \alpha_{\kappa}$.
	First, it is easily checked that $\sol^1 \in \csol$ and $\lambda^\circ, \lambda^1 \in \cev$ for $\ctilde$ sufficiently small.
	Further, since $\sigma \alpha_{\kappa}^{-1} \gamma_{\kappa}^{-1} \leq 1$ by definition of $\phi_{\epsilon,\sigma}$,
	\begin{equation*}
		\Vert \solcirc - \sol^1 \Vert_2^2 = (\solcirc_{\kappa} - \sol^1_{\kappa})^2 = \frac{4\ctilde^2\sigma^2  \alpha_{\kappa}^{-2}\gamma_{\kappa}^{-2}  }{1 + \ctilde\sigma \alpha_{\kappa}^{-1} \gamma_{\kappa}^{-1}} \cdot (\solcirc_{\kappa})^2 \geq \frac{4\ctilde^2}{1 + \ctilde} \cdot (\solcirc_{\kappa})^2 \cdot \phi_{\epsilon, \sigma}^2 \eqdef c^2 \phi_{\epsilon, \sigma}^2,
	\end{equation*}
	and $c \to 0$ if and only if $\ctilde \to 0$, showing that $\sol^1 \in \setAlt(c\phi_{\epsilon,\sigma})$ for $\ctilde$ sufficiently small.
	Thus, it remains to show that the Kullback-Leibler distance between the two hypotheses can be made arbitrary small by choosing the parameter $\ctilde$ sufficiently small.
	By construction, $\KL(\Pb^{X,Y}_0, \Pb^{X,Y}_1) = \KL(\Pb^{Y_{\kappa}}_0, \Pb^{Y_{\kappa}}_1)$, and hence
	\begin{align*}
		\KL(\Pb^{X,Y}_0, \Pb^{X,Y}_1) = \frac{2}{\sigma^2} \ctilde^2 \sigma^2 \gamma_{\kappa}^{-2} \alpha_{\kappa}^{-2} \alpha_{\kappa}^2 \leq 2\ctilde^2,
	\end{align*}
	and $2\ctilde^2 \leq 2\delta^2 \Leftrightarrow \ctilde \leq \delta$ implies the claim assertion grant to Statement~\ref{lemLowerTesting2} of Lemma~\ref{lemLowerTesting} with $\mu = \delta_{(\sol^1, \lambda^1)}$. 
\printbibliography

\end{document}